\documentclass[a4paper]{amsart}

\usepackage{amsmath,amsrefs,amssymb}
\usepackage[mathscr]{euscript}
\usepackage[all]{xy}
\newdir{ >}{{}*!/-5pt/@{>}} %% cf xyguide exercise 14
\SelectTips{cm}{} %% cf xyguide page 9
\usepackage{tikz}

\title{Kan subdivision and products of simplicial sets}
\author{Vegard Fjellbo}
\address{Department of Mathematics, University of Oslo, Norway}
\email{rvfjellb@math.uio.no}
\author{John Rognes}
\address{Department of Mathematics, University of Oslo, Norway}
\email{rognes@math.uio.no} \urladdr{http://folk.uio.no/rognes}
\date{June 6th 2014}

\newtheorem{theorem}{Theorem}[section]
\newtheorem{proposition}[theorem]{Proposition}
\newtheorem{lemma}[theorem]{Lemma}
\newtheorem{corollary}[theorem]{Corollary}
\theoremstyle{definition}
\newtheorem{definition}[theorem]{Definition}
\newtheorem{notation}[theorem]{Notation}
\theoremstyle{remark}
\newtheorem{example}[theorem]{Example}
\newtheorem{remark}[theorem]{Remark}

\renewcommand{\:}{\colon}
\newcommand{\longto}{\longrightarrow}
\DeclareMathOperator{\im}{im}
\DeclareMathOperator{\In}{in}
\DeclareMathOperator{\pr}{pr}
\DeclareMathOperator{\Sd}{Sd}
\DeclareMathOperator{\red}{red}
\newcommand{\sC}{\mathscr{C}}
\newcommand{\sD}{\mathscr{D}}

\begin{document}
\begin{abstract}
The canonical map from the Kan subdivision of a product of finite
simplicial sets to the product of the Kan subdivisions is a simple map, in
the sense that its geometric realization has contractible point inverses.
\end{abstract}
\maketitle{}

\section{Introduction}

\noindent
Kan's normal subdivision \cite{Kan}*{\S7} is a functor $\Sd$ from
simplicial sets to simplicial sets.  It agrees with barycentric
subdivision when applied to (ordered) simplicial complexes.  The maps
induced by applying Kan subdivision to the projections $X \times Y \to
X$ and $X \times Y \to Y$ combine to a canonical map $\kappa \: \Sd(X
\times Y)\to \Sd X \times \Sd Y$.

\begin{definition}
A map $f \: A \to B$ of finite simplicial sets is said to be
\emph{simple} \cite{WJR}*{2.1.1} if its geometric realization $|f| \:
|A| \to |B|$ has contractible point inverses, meaning that the preimage
$|f|^{-1}(b)$ is contractible for each point $b \in |B|$.  We write
$A\xrightarrow{\simeq_s} B$ to denote a simple map.
\end{definition}

\begin{theorem} \label{theorem_main_result}
Let $X$ and $Y$ be finite simplicial sets.  The canonical map
$$
\kappa \: \Sd(X \times Y) \xrightarrow{\simeq_s} \Sd X \times \Sd Y \,,
$$
from the Kan subdivision of the product $X \times Y$ to the
product of the Kan subdivisions, is a simple map.
\end{theorem}

The theorem follows from the special cases when $X = \Delta[m]$ is a
standard simplex, for some $m\ge0$, by an induction on the dimension
and number of top-dimensional cells in the CW complex $|X|$.  A second
induction, over the cells of $|Y|$, allows us to specialize further to
the cases when $Y = \Delta[n]$, for some $n\ge0$.  Hence our real task
is to prove the following result.

\begin{proposition} \label{prop_main_result}
The canonical map
$$
\kappa \: \Sd(\Delta[m] \times \Delta[n])
	\xrightarrow{\simeq_s} \Sd \Delta[m] \times \Sd \Delta[n]
$$
is simple, for each $m\ge0$ and $n\ge0$.
\end{proposition}

\begin{example}
For $m=n=1$, we have the following picture.
$$
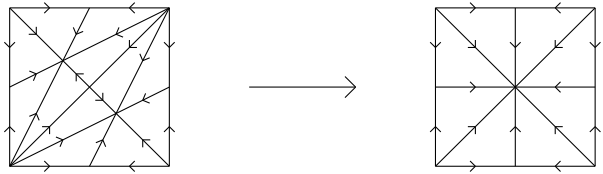
\begin{tikzpicture}[scale=0.7, transform shape]

	\draw [thin] (-1.5,1.5) -- (-1.5,-1.5);
	\draw [thin] (-1.5,-0.75) -- (-1.6,-0.85);
	\draw [thin] (-1.5,-0.75) -- (-1.4,-0.85);
	\draw [thin] (-1.5,0.75) -- (-1.6,0.85);
	\draw [thin] (-1.5,0.75) -- (-1.4,0.85);

	\draw [thin] (-1.5,-1.5) -- (1.5,-1.5);
	\draw [thin] (-0.75,-1.5) -- (-0.85,-1.6);
	\draw [thin] (-0.75,-1.5) -- (-0.85,-1.4);
	\draw [thin] (0.75,-1.5) -- (0.85,-1.6);
	\draw [thin] (0.75,-1.5) -- (0.85,-1.4);

	\draw [thin] (1.5,-1.5) -- (1.5,1.5);
	\draw [thin] (1.5,-0.75) -- (1.4,-0.85);
	\draw [thin] (1.5,-0.75) -- (1.6,-0.85);
	\draw [thin] (1.5,0.75) -- (1.4,0.85);
	\draw [thin] (1.5,0.75) -- (1.6,0.85);

	\draw [thin] (1.5,1.5) -- (-1.5,1.5);
	\draw [thin] (-0.75,1.5) -- (-0.85,1.6);
	\draw [thin] (-0.75,1.5) -- (-0.85,1.4);
	\draw [thin] (0.75,1.5) -- (0.85,1.6);
	\draw [thin] (0.75,1.5) -- (0.85,1.4);

	\draw [thin] (-1.5,-1.5) -- (1.5,1.5);
	\draw [thin] (-0.75,-0.75) -- (-0.75,-0.8914);
	\draw [thin] (-0.75,-0.75) -- (-0.8914,-0.75);
	\draw [thin] (0.75,0.75) -- (0.75,0.8914);
	\draw [thin] (0.75,0.75) -- (0.8914,0.75);

	\draw [thin] (-1.5,1.5) -- (1.5,-1.5);
	\draw [thin] (-1,1) -- (-1.1414,1);
	\draw [thin] (-1,1) -- (-1,1.1414);
	\draw [thin] (1,-1) -- (1.1414,-1);
	\draw [thin] (1,-1) -- (1,-1.1414);
	\draw [thin] (-0.25,0.25) -- (-0.1086,0.25);
	\draw [thin] (-0.25,0.25) -- (-0.25,0.1086);
	\draw [thin] (0.25,-0.25) -- (0.1086,-0.25);
	\draw [thin] (0.25,-0.25) -- (0.25,-0.1086);

	\draw [thin] (-1.5,-1.5) -- (0,1.5);
	\draw [thin] (-1,-0.5) -- (-1.1306,-0.5541);
	\draw [thin] (-1,-0.5) -- (-0.9459,-0.6306);
	\draw [thin] (-0.25,1) -- (-0.3041,1.1306);
	\draw [thin] (-0.25,1) -- (-0.1194,1.0541);

	\draw [thin] (0,-1.5) -- (1.5,1.5);
	\draw [thin] (0.25,-1) -- (0.3041,-1.1306);
	\draw [thin] (0.25,-1) -- (0.1194,-1.0541);
	\draw [thin] (1,0.5) -- (1.1306,0.5541);
	\draw [thin] (1,0.5) -- (0.9459,0.6306);

        \draw [thin] (-1.5,0) -- (1.5,1.5);
	\draw [thin] (0.5,1) -- (0.5541,1.1306);
	\draw [thin] (0.5,1) -- (0.6306,0.9459);
	\draw [thin] (-1,0.25) -- (-1.1306,0.3041);
	\draw [thin] (-1,0.25) -- (-1.0541,0.1194);

	\draw [thin] (-1.5,-1.5) -- (1.5,0);
	\draw [thin] (-0.5,-1) -- (-0.5541,-1.1306);
	\draw [thin] (-0.5,-1) -- (-0.6306,-0.9459);
	\draw [thin] (1,-0.25) -- (1.1306,-0.3041);
	\draw [thin] (1,-0.25) -- (1.0541,-0.1194);

	\draw [thin] (3,0) -- (5,0);
	\draw [thin] (4.8,0.2) -- (5,0);
	\draw [thin] (4.8,-0.2) -- (5,0);

	\draw [thin] (6.5,1.5) -- (9.5,1.5);
	\draw [thin] (6.5,1.5) -- (6.5,-1.5);
	\draw [thin] (6.5,-1.5) -- (9.5,-1.5);
	\draw [thin] (9.5,-1.5) -- (9.5,1.5);

	\draw [thin] (6.5,-0.75) -- (6.4,-0.85);
	\draw [thin] (6.5,-0.75) -- (6.6,-0.85);
	\draw [thin] (6.5,0.75) -- (6.4,0.85);
	\draw [thin] (6.5,0.75) -- (6.6,0.85);

	\draw [thin] (7.25,-1.5) -- (7.15,-1.6);
	\draw [thin] (7.25,-1.5) -- (7.15,-1.4);
	\draw [thin] (8.75,-1.5) -- (8.85,-1.6);
	\draw [thin] (8.75,-1.5) -- (8.85,-1.4);

	\draw [thin] (9.5,-0.75) -- (9.4,-0.85);
	\draw [thin] (9.5,-0.75) -- (9.6,-0.85);
	\draw [thin] (9.5,0.75) -- (9.4,0.85);
	\draw [thin] (9.5,0.75) -- (9.6,0.85);

	\draw [thin] (7.25,1.5) -- (7.15,1.6);
	\draw [thin] (7.25,1.5) -- (7.15,1.4);
	\draw [thin] (8.75,1.5) -- (8.85,1.6);
	\draw [thin] (8.75,1.5) -- (8.85,1.4);

	\draw [thin] (6.5,1.5) -- (9.5,-1.5);
	\draw [thin] (6.5,-1.5) -- (9.5,1.5);

	\draw [thin] (8.75,-0.75) -- (8.8914,-0.75);
	\draw [thin] (8.75,-0.75) -- (8.75,-0.8914);

	\draw [thin] (8.75,0.75) -- (8.8914,0.75);
	\draw [thin] (8.75,0.75) -- (8.75,0.8914);

	\draw [thin] (7.25,-0.75) -- (7.25,-0.8914);
	\draw [thin] (7.25,-0.75) -- (7.1086,-0.75);

	\draw [thin] (7.25,0.75) -- (7.25,0.8914);
	\draw [thin] (7.25,0.75) -- (7.1086,0.75);

	\draw [thin] (6.5,0) -- (9.5,0);
	\draw [thin] (8.75,0) -- (8.85,0.1);
	\draw [thin] (8.75,0) -- (8.85,-0.1);
	\draw [thin] (7.25,0) -- (7.15,0.1);
	\draw [thin] (7.25,0) -- (7.15,-0.1);

	\draw [thin] (8,-1.5) -- (8,1.5);
	\draw [thin] (8,-0.75) -- (7.9,-0.85);
	\draw [thin] (8,-0.75) -- (8.1,-0.85);
	\draw [thin] (8,0.75) -- (7.9,0.85);
	\draw [thin] (8,0.75) -- (8.1,0.85);

\end{tikzpicture}
$$
The map $\kappa$ is the identity on the boundary, and takes the three
interior vertices on the left hand side to the single interior vertex
on the right hand side.  The rhombus on the left is collapsed to the
diagonal on the right.  The point inverses of $|\kappa|$ are points or
closed intervals.  Hence $|\kappa|$ is not a homeomorphism, but all
preimages of points are contractible. Note that $|\kappa|$ does not
admit a continuous section.
\end{example}

\begin{remark}
Simple maps have a direction: given a simple map $f \: A \to B$ there
might be no simple map $g \: B \to A$.  Simple-homotopy theory \cite{Coh}
concerns the equivalence relation on finite simplicial sets (or simplicial
complexes) generated by simple maps.  Piecewise-linear topology \cite{RS}
concerns the properties of simplicial complexes that are invariant under
linear subdivisions.  It is often convenient and sufficient to only consider
the linear subdivisions that arise by iterated barycentric subdivision
\cite{Spa}*{\S3.5, \S3.6}.  Our theorem gives a strong version of the
statement that the simple-homotopy type of a product is a well-defined
piecewise-linear notion, also with respect to this more restricted notion
of subdivision.

The weaker statement that the map $\kappa$ in the theorem is a
simple-homotopy equivalence is easily proved: there is a natural
\emph{last vertex map} $d_X \: \Sd X \to X$, which is simple for all
finite $X$ \cite{WJR}*{2.2.17}, and $(d_X \times d_Y) \circ \kappa =
d_{X \times Y}$.  However, this argument does not suffice to prove that
$\kappa$ is a simple map.
\end{remark}

In Section~\ref{sec_improvement} we present an application of our
main result.  In Section~\ref{sec_outline} we show how to deduce
Theorem~\ref{theorem_main_result} from Proposition~\ref{prop_main_result},
and outline the proof of the latter result.  The details of that
proof occupy Sections~\ref{sec_posets_of_paths}, \ref{sec_map_cyl}
and~\ref{sec_contracting}.

\section{An application to the improvement functor} \label{sec_improvement}

\noindent
A simplicial set $X$ is said to be \emph{non-singular} \cite{WJR}*{1.2.2}
if for each non-degenerate $n$-simplex $x$ in $X$ the representing
map $\bar x \: \Delta[n] \to X$ is a cofibration, or equivalently,
if each non-degenerate $n$-simplex has $n+1$ distinct vertices, for
each $n\ge0$.  This implies that the geometric realization $|X|$ has
a preferred piecewise-linear structure.  Let $\sC$ be the category of finite
simplicial sets and simplicial maps, let $\sD$ be its full subcategory
of finite non-singular simplicial sets, and let $s\sC$ and $s\sD$ denote
the respective subcategories of simple maps.

Our main result has an application to the multiplicative properties
of the \emph{improvement functor} $I \: \sC \to \sD$ constructed in
\cite{WJR}*{2.5.2}.  This functor associates to each finite simplicial set
$X$ a finite non-singular simplicial set $I(X)$, together with a natural
simple map $s_X \: I(X) \to X$.  Hence the restricted functor $sI \:
s\sC \to s\sD$ induces a homotopy equivalence of classifying spaces,
and the localized functor $I[s^{-1}] \: \sC[s^{-1}] \to \sD[s^{-1}]$
(inverting the respective collections of simple maps \cite{GZ}*{1.1}) is an
equivalence of categories.  This shows that the simple-homotopy theory
of finite simplicial sets is equivalent to the simple-homotopy theory
of finite non-singular simplicial sets.  For a useful relative version
of this statement, see \cite{WJR}*{1.2.5}.  The application we have in
mind concerns the compatibility of this equivalence with the categorical
products.

\begin{proposition} \label{proposition_improvement}
Let $X$ and $Y$ be finite simplicial sets, and let $I$ denote the
improvement functor from \cite{WJR}*{\S2.5}.  The canonical map
$$
I(X\times Y) \xrightarrow{\simeq_s} I(X)\times I(Y) \,,
$$
induced by the projections $X \times Y \to X$ and $X \times Y \to Y$,
is simple.
\end{proposition}

Before showing how to deduce this from our main theorem, we recall
some notation and terminology.  Each simplex $x$ of a simplicial set
$X$ is a degeneration of a unique non-degenerate simplex, which we
denote $x^\#$.  If $x$ is non-degenerate then $x = x^\#$.  Let $X^\#$
be the set of non-degenerate simplices in $X$, partially ordered by
letting $x \le y$ if $x$ is a face of $y$.  We call the nerve $B(X)
= N(X^\#)$ of this partially ordered set the \emph{Barratt nerve} of
$X$, due to its early appearance in \cite{Ba}.  A simplicial map $f \:
X \to Y$ induces an order-preserving function $f^\# \: X^\# \to Y^\#$,
taking $x$ to $f(x)^\#$, and a functorial map of nerves $B(f) = N(f^\#)
\: B(X) \to B(Y)$.  The Kan subdivision $\Sd(X)$ is defined as the left
Kan extension of $[n] \mapsto B(\Delta^n)$ along the Yoneda embedding of
$\Delta$ into simplicial sets \cite{WJR}*{2.2.7}, so there is a natural
map $b_X \: \Sd(X) \to B(X)$, which is an isomorphism for non-singular
$X$ \cite{WJR}*{2.2.11}.  The improvement functor $I$ is defined as
the composite
$$
I(X) = B(\Sd(X)) \,.
$$
(The use of the opposite subdivision in \cite{WJR} plays no explicit
role in our arguments, and will be suppressed.)

\begin{remark}
The improvement functor $I$ appears implicitly in the work of Thomason
\cite{Th}, as the composite $N \circ c \Sd^2$ of the nerve functor $N$
(from small categories to simplicial sets) and the categorified double
subdivision $c \Sd^2$ (from simplicial sets to categories).  The latter
functor is the left adjoint in Thomason's Quillen equivalence between
these two (closed) model categories.  By \cite{Th}*{Lemma~5.6} the
category $c \Sd^2(X)$ is a partially ordered set, for any simplicial set
$X$, and this makes it easy to identify it with the partially ordered
set $\Sd(X)^\#$ of non-degenerate simplices in $\Sd(X)$.
\end{remark}

\begin{proof}[Proof of Proposition~\ref{proposition_improvement}]
There are canonical maps
$$
\Sd(\Sd(X \times Y)) \to \Sd(\Sd(X) \times \Sd(Y))
	\to \Sd(\Sd(X)) \times \Sd(\Sd(Y)) \,.
$$
The left hand map, $\Sd(\kappa)$, is simple by our
Theorem~\ref{theorem_main_result}, combined with the fact that Kan
subdivision preserves simple maps \cite{WJR}*{2.3.3}.  The right hand
map is simple by the same theorem applied for the finite simplicial
sets $\Sd(X)$ and $\Sd(Y)$.  Hence the composite map is also simple
\cite{WJR}*{2.1.3(a)}.

The natural map $b_{\Sd(X)} \: \Sd(\Sd(X)) \to B(\Sd(X))$ is simple by
\cite{WJR}*{2.5.5, 2.5.8}.  Hence the vertical maps are simple in the
commutative square
$$
\xymatrix{
\Sd(\Sd(X \times Y)) \ar[rr]^-{\simeq_s} \ar[d]_-{\simeq_s}
	&& \Sd(\Sd(X)) \times \Sd(\Sd(Y)) \ar[d]^-{\simeq_s} \\
B(\Sd(X \times Y)) \ar[rr] && B(\Sd(X)) \times B(\Sd(Y)) \,.
}
$$
We have just argued that the upper horizontal map is simple,
and this implies that the lower horizontal map is simple, by
the right cancellation property of simple maps \cite{WJR}*{2.1.3(b)}.
\end{proof}

\section{Outline of argument} \label{sec_outline}

\noindent
\begin{proof}[Proof of Theorem~\ref{theorem_main_result} assuming
Proposition~\ref{prop_main_result}]
To reduce Theorem~\ref{theorem_main_result} to the case $X = \Delta[m]$,
we write $X = \Delta[m] \cup_{\partial\Delta[m]} X'$ and think of
$\kappa \: \Sd(X \times Y) \to \Sd X \times \Sd Y$ as the vertical map
of horizontal pushouts in the diagram
$$
\xymatrix{
\Sd(\Delta[m] \times Y) \ar[d]_-{\kappa}
& \Sd(\partial\Delta[m] \times Y) \ar[d]_-{\kappa} \ar@{ >->}[l] \ar[r]
& \Sd(X' \times Y) \ar[d]^-{\kappa} \\
\Sd \Delta[m] \times \Sd Y
& \Sd \partial\Delta[m] \times \Sd Y \ar@{ >->}[l] \ar[r]
& \Sd X' \times \Sd Y \,.
}
$$
This uses the fact that $\Sd$ preserves colimits and cofibrations
\cite{WJR}*{2.2.9}.  By induction we may assume that $\kappa$ is simple
for $\partial\Delta[m]$ and $X'$.  The case of general $X$ then
follows from the case $X = \Delta[m]$, by the gluing lemma for simple
maps \cite{WJR}*{2.1.3(d)}.

The reduction to the case $Y = \Delta[n]$ is similar.  Write
$Y = \Delta[n] \cup_{\partial\Delta[n]} Y'$ and think of $\kappa \:
\Sd(\Delta[m] \times Y) \to \Sd \Delta[m] \times \Sd Y$ as the map
of pushouts in the diagram
$$
\xymatrix{
\Sd(\Delta[m] \times \Delta[n]) \ar[d]_-{\kappa}
& \Sd(\Delta[m] \times \partial\Delta[n]) \ar[d]_-{\kappa} \ar@{ >->}[l] \ar[r]
& \Sd(\Delta[m] \times Y') \ar[d]^-{\kappa} \\
\Sd \Delta[m] \times \Sd \Delta[n]
& \Sd \Delta[m] \times \Sd \partial\Delta[n] \ar@{ >->}[l] \ar[r]
& \Sd \Delta[m] \times \Sd Y' \,.
}
$$
By induction we may assume that $\kappa$ is simple for $\partial\Delta[n]$
and $Y'$, so the main theorem follows from the case $Y = \Delta[n]$,
i.e., from Proposition~\ref{prop_main_result}.
\end{proof}

\begin{proof}[Outline of proof of Proposition~\ref{prop_main_result}]
To show that the canonical map
$$
\kappa \: \Sd(\Delta[m] \times \Delta[n])
	\longto \Sd \Delta[m] \times \Sd \Delta[n]
$$
is simple, it suffices to show that it is simple over the interior of each
non-degenerate $r$-simplex $(z,w) = (z_0 \le \dots \le z_r, w_0 \le \dots
\le w_r)$ of the target, for $r\ge0$, cf.~Notation~\ref{not_zw_posets_Pi}.
By Lemmas~\ref{lemma_preimage_kappa},
\ref{lemma_description_of_pointinverses} and~\ref{lem_xi_pi_coincide},
$\kappa$ agrees with the nerve of the order-preserving function $\pi \:
P(\phi_r, \dots, \phi_1) \to [r]$ over that interior.  Here
$$
F_r \overset{\phi_r}\longto \dots \overset{\phi_1}\longto F_0
$$
is a diagram of partially ordered sets and order-preserving functions,
specified in Notations~\ref{not_partiallyorderedsetsofpaths}
and~\ref{not_restricting_functions_paths},
and $P(\phi_r, \dots, \phi_1)
= F_r \sqcup_{\phi_r} \dots \sqcup_{\phi_1} F_0$ is as in
Definition~\ref{def_iter_mappingcylinder_partiallyorderedsets}.

Let
$$
NF_r \overset{f_r}\longto \dots \overset{f_1}\longto NF_0
$$
be the induced diagram of nerves.  According to
Definition~\ref{def_iterated_mapping_cylinder}, the nerve of $\pi$
equals the reduced coordinate projection $\pi \: M(f_r, \dots,
f_1) \to \Delta[r]$, where $M(f_r, \dots, f_1)$ is the $r$-fold
iterated \emph{reduced mapping cylinder}.  As we explain in
Remark~\ref{remark_outline_of_proof}, there is a commutative diagram
$$
\xymatrix{
T(f_r, \dots, f_1) \ar[r]^-{\red} \ar[d] & M(f_r, \dots, f_1) \ar[d]^{\pi} \\
T^r \ar[r]^-{\red} & \Delta[r] \rlap{\,,}
}
$$
where $T(f_r, \dots, f_1)$ is the $r$-fold iterated ordinary mapping
cylinder, cf.~Definition~\ref{def_iterated_ordinary_mapping_cylinder}.
The reduction map $\red \: T(f_r, \dots, f_1) \to M(f_r, \dots, f_1)$
is simple by Lemma~\ref{lemma_cylinder_reduction_map_is_simple},
and the map $\red \: T^r \to \Delta[r]$ is simple by
Lemma~\ref{lemma_cylinder_reduction_between_terminal_ordcyl_and_redcyl}.

By the composition and right cancellation properties of simple
maps~\cite{WJR}*{2.1.3}, the reduced coordinate projection $\pi$ will be
simple if the ordinary cylinder projection $T(f_r, \dots, f_1) \to T^r$
is simple.  Each point inverse of the geometric realization $|T(f_r,
\dots, f_1)| \to |T^r|$ is homeomorphic to one of the spaces
$|NF_i|$, for $0 \le i \le r$, so it suffices to prove that each of
the partially ordered sets $F_i$ has contractible classifying space.
Using Notation~\ref{not_pzw}, each $F_i$ is of the form $P^{z',w'}$ for
a suitable $i$-simplex $(z',w')$ of $\Sd \Delta[m] \times \Sd \Delta[n]$.
Hence the required contractibility is a consequence of our final technical
result, Proposition~\ref{prop_npzw_is_contractible}.
\end{proof}

\section{Partially ordered sets of paths} \label{sec_posets_of_paths}

\noindent
As a first step towards the proof of Proposition~\ref{prop_main_result},
we unravel the definition of the canonical map $\kappa$ in that special
case, and identify the part of $\kappa$ that sits over a non-degenerate
$r$-simplex in its target.

The $m$-simplex $\Delta[m] = N([m])$ is the nerve of the totally
ordered set $[m] = \{0 < 1 < \dots < m\}$.  Its $k$-simplices are the
order-preserving functions $[k] \to [m]$, which we will refer to as
\emph{operators}, following \cite{FP}*{\S4.1}.  Injective operators
are called \emph{face operators}, and surjective operators are
called \emph{degeneracy operators}.  The non-degenerate simplices
of $\Delta[m]$ are the face operators $\mu \: [k] \to [m]$, which
correspond to the non-empty subsets $\im(\mu)$ of $[m]$.  As in
Section~\ref{sec_improvement}, we partially order the set $\Delta[m]^\#$
of non-degenerate simplices by setting $\mu \le \zeta$ if $\mu$ is a face
of $\zeta$, or equivalently, if the image of $\mu$ is a subset of the
image of $\zeta$.  The Kan subdivision $\Sd \Delta[m] = N(\Delta[m]^\#)$
is the nerve of that partially ordered set.  Hence an $r$-simplex of
$\Sd \Delta[m]$ is a chain $\zeta_0 \le \zeta_1 \le \dots \le \zeta_r$
of face operators $\zeta_i \: [k_i] \to [m]$ for $0 \le i \le r$.

The product $\Delta[m] \times \Delta[n]$ can be regarded as the nerve of
$[m] \times [n]$ with the product partial ordering.  Its $k$-simplices
are the order-preserving functions $[k] \to [m] \times [n]$, and its
non-degenerate $k$-simplices are the injective order-preserving functions
$$
\gamma \: [k] \to [m] \times [n] \,,
$$
which correspond to the non-empty totally ordered subsets $\im(\gamma)$
of $[m] \times [n]$.  The set $C = (\Delta[m] \times \Delta[n])^\#$ of
non-degenerate simplices is partially ordered by setting $\beta \le
\gamma$ if $\beta$ is a face of $\gamma$, or equivalently, if the image
of $\beta$ is a subset of the image of~$\gamma$.  The product $\Delta[m]
\times \Delta[n]$ is a non-singular simplicial set, so its Kan subdivision
$\Sd(\Delta[m] \times \Delta[n]) = N((\Delta[m] \times \Delta[n])^\#)$
is the nerve of the partially ordered set of its non-degenerate simplices.
An $r$-simplex is a chain $\gamma_0 \le \gamma_1 \le \dots \le \gamma_r$
of injective order-preserving functions $\gamma_i \: [k_i] \to [m] \times
[n]$ for $0\le i\le r$.

The first projection $\pr_1 \: \Delta[m] \times \Delta[n] \to \Delta[m]$
induces the order-preserving function $\pr_1^\#$ that takes each
injective order-preserving function $\gamma \: [k] \to [m]\times [n]$
to the non-degenerate part $(\pr_1 \circ \gamma)^\# \: [m_1] \to [m]$ of
the composite $\pr_1 \circ \gamma$.  This is the non-degenerate simplex
with image $\im(\pr_1^\#(\gamma))=\im(\pr_1 \circ \gamma)$. Similarly
the second projection $\pr_2\: \Delta[m] \times \Delta[n] \to \Delta[n]$
induces an order-preserving function $\pr_2^\#$ that takes $\gamma$ to
the non-degenerate simplex $[n_1]\rightarrow [n]$ with image $\im(\pr_2
\circ \gamma)$.
We can display these order-preserving functions in the following diagram,
where a feathered arrow denotes an injection and a two-headed arrow
denotes a surjection.
$$
\xymatrix{
[m_1] \ar@{ >->}[d]_{\pr_1^\#(\gamma)}
	& [k] \ar@{ >->}[d]_{\gamma} \ar@{->>}[l] \ar@{->>}[r]
	& [n_1] \ar@{ >->}[d]^{\pr_2^\#(\gamma)} \\
[m] & [m] \times [n] \ar[l]_-{\pr_1} \ar[r]^-{\pr_2} & [n]
}
$$
The functorial map $\Sd(\pr_1) \: \Sd(\Delta[m] \times \Delta[n]) \to \Sd
\Delta[m]$ equals the map of nerves $N(\pr_1^\#)$ induced by $\pr_1^\#$,
and similarly for the second projection.  Hence the canonical map $\kappa
= (\Sd(\pr_1), \Sd(\pr_2))$ equals the map $(N(\pr_1^\#), N(\pr_2^\#))$,
which in turn can be identified with the map of nerves
induced by
$$
(\pr_1^\#, \pr_2^\#) \: (\Delta [m]\times \Delta [n])^\#
	\to \Delta [m]^\# \times \Delta [n]^\# \,,
$$
since the diagram
$$
\xymatrix{
N((\Delta [m]\times \Delta [n])^\#) \ar[d]_{N(\pr_1^\#, \pr_2^\#)}
  \ar[dr]^-{\kappa} \\
N(\Delta [m]^\# \times \Delta [n]^\#) \ar[r]_-{\cong}
  & N(\Delta [m]^\#)\times N(\Delta [n]^\#)
}
$$
commutes.

\begin{definition}
Let $\mu \in \Delta[m]^\#$ and $\nu \in \Delta[n]^\#$ be non-degenerate
simplices.  A non-degenerate simplex $\gamma \in (\Delta[m] \times
\Delta[n])^\#$ maps under $(\pr_1^\#, \pr_2^\#)$ to $(\mu, \nu)$ if and
only if $\im(\pr_1 \circ \gamma) = \im(\mu)$ and $\im(\pr_2 \circ \gamma)
= \im(\nu)$.  In this case we say that $\gamma$ is a \emph{$(\mu,
\nu)$-path}.  We write $P^{\mu, \nu}$ for the set of $(\mu, \nu)$-paths,
partially ordered as a subset of $C = (\Delta[m] \times \Delta[n])^\#$.
\end{definition}

This terminology (being a $(\mu, \nu)$-path) will be important throughout
this paper.  Note that being a $(\mu, \nu)$-path is a stronger condition
than just being a path with image in $\im(\mu) \times \im(\nu)$.

\begin{notation} \label{not_zw_posets_Pi}
Let $(z, w) = (z_0 \le \dots \le z_r, w_0 \le \dots \le w_r)$ be a
non-degenerate $r$-simplex in $\Sd \Delta[m] \times \Sd \Delta[n]$, with
characteristic map $\overline{(z, w)} \: \Delta[r] \to \Sd \Delta[m]
\times \Sd \Delta[n]$, so that $z_i \in \Delta[m]^\#$ and $w_i \in
\Delta[n]^\#$ are non-degenerate simplices for $0 \le i \le r$.  Let $P_i
= P^{z_i, w_i}$ denote the partially ordered set of $(z_i, w_i)$-paths.

The $(z_i, w_i)$ are pairwise distinct, so the $P_i$ are pairwise
disjoint as sets, for $0 \le i \le r$.  Let $P_0 \cup \dots \cup P_r$
denote the union of these sets, partially ordered as a subset of $C =
(\Delta[m] \times \Delta[n])^\#$, and let $\lambda \: P_0 \cup \dots
\cup P_r \to [r]$ be the order-preserving function taking the elements
of $P_i$ to $i$, for each $0 \le i \le r$.
\end{notation}

By the restriction of a map $f \: A \to B$ over a simplicial subset $Z
\subseteq B$, we mean the associated map $f^{-1}(Z) \to Z$, given by
restricting $f$ to $f^{-1}(Z) \subseteq A$, and corestricting the target
to $Z$.

\begin{lemma} \label{lemma_preimage_kappa}
The restriction of $\kappa$ over the simplicial subset of $\Sd \Delta[m]
\times \Sd \Delta[n]$ generated by $(z, w)$ is isomorphic to the map
$N\lambda \: N(P_0 \cup \dots \cup P_r) \to \Delta[r]$.
\end{lemma}

\begin{proof}
Nerves commute with pullbacks, which means we have a pullback square
$$
\xymatrix{
N(P_0 \cup \dots \cup P_r) \ar@{ >->}[r] \ar[d]_{N\lambda}
	& \Sd(\Delta[m] \times \Delta[n]) \ar[d]^{\kappa} \\
\Delta[r] \ar@{ >->}[r]^-{\overline{(z, w)}}
	& \Sd \Delta[m] \times \Sd \Delta[n] \,.
}
$$
\end{proof}

\begin{corollary} \label{corollary_description_pointinverses_of_0-cells}
Let $(\mu, \nu)$ be a $0$-simplex in $\Sd \Delta[m] \times \Sd \Delta[n]$.
The preimage under $\kappa$ of the simplicial subset generated by $(\mu,
\nu)$ equals the nerve of the partially ordered set $P^{\mu,\nu}$
of $(\mu, \nu)$-paths.  Hence the preimage under $|\kappa|$ of the
corresponding $0$-cell in $|\Sd \Delta[m] \times \Sd \Delta[n]|$ equals
the classifying space $|NP^{\mu, \nu}|$ of that partially ordered set.
\end{corollary}

We now explain how only a part of the union $P_0 \cup \dots \cup P_r$
plays a role over the interior of the $r$-cell $|\Delta[r]|$.

\begin{definition}
Let $(\mu, \nu)$, $(\zeta, \eta) \in \Delta [m]^\# \times \Delta
[n]^\#$ be such that $\mu \le \zeta$ and $\nu \le \eta$, and let $L
= \im(\mu) \times \im(\nu) \subseteq [m] \times [n]$.  We say that a
$(\zeta, \eta)$-path $\gamma \: [k] \to [m] \times [n]$ is \emph{$(\mu,
\nu)$-full} if some face $\beta$ of $\gamma$ is a $(\mu, \nu)$-path.
In this case there is a greatest such face $\beta$, with image $\im(\beta)
= \im(\gamma) \cap L$.  We write $\beta = \gamma \cap L$ for this greatest
$(\mu, \nu)$-path.
\end{definition}

This terminology (being $(\mu, \nu)$-full) will also be important
throughout this paper.  In general, not every $(\zeta, \eta)$-path will
be $(\mu, \nu)$-full.

\begin{notation} \label{not_partiallyorderedsetsofpaths}
For each $0 \le i \le r$, let $F_i$ denote the set of $(z_i, w_i)$-paths
that are $(z_j, w_j)$-full for every $0 \le j < i$.  Partially order $F_i$
as a subset of $P_i$, and partially order $F_0 \cup \dots \cup F_r$ as
a subset of $P_0 \cup \dots \cup P_r$.  The $F_i$ are pairwise disjoint
as sets, since this holds for the $P_i$.
Let $\xi \: F_0 \cup \dots \cup F_r \to [r]$ be the order-preserving
function that takes each element of $F_i$ to $i$.
\end{notation}

The notations $P_i$ and $F_i$ are chosen for their brevity.
They always depend on the implicit choice of a non-degenerate $r$-simplex
$(z, w)$, as above.

\begin{lemma} \label{lemma_description_of_pointinverses}
The maps $|N\xi|$ and $|N\lambda|$ agree over the interior of $|\Delta
[r]|$.
\end{lemma}

\begin{proof}
By the definition of $\xi$ and $\lambda$, the triangle
$$
\xymatrix{
N(F_0 \cup \dots \cup F_r) \ar[dr]_{N\xi} \ar@{ >->}[r]
	& N(P_0 \cup \dots \cup P_r) \ar[d]^{N\lambda} \\
& \Delta[r]
}
$$
commutes. Let $x \in N(P_0 \cup \dots \cup P_r)$ be a $p$-simplex that
does not lie over the boundary $\partial \Delta[r]$ of $\Delta[r]$.
In other words, $x$ is an order-preserving function such that the
composite
$$
[p] \xrightarrow{x} P_0 \cup \dots \cup P_r \xrightarrow{\lambda} [r]
$$
is surjective.  We claim that $x$ factors through $F_0 \cup \dots \cup
F_r$.  To prove this, let $u \in [p]$ be arbitrary, fix $i$ so that
$x(u) \in P_i$, and consider any $0 \le j < i$.  The surjectivity of
$\lambda \circ x$ tells us that there is a $v \in [p]$ with $x(v) \in P_j$.
We cannot have $v \ge u$, since $\lambda \circ x$ is order-preserving,
so $v < u$ and the $(z_j, w_j)$-path $x(v)$ must be a face of $x(u)$.
This shows that $x(u)$ is $(z_j, w_j)$-full, which implies that $x(u)
\in F_i$.

It follows that $|N(F_0 \cup \dots \cup F_r)| \to |N(P_0 \cup \dots
\cup P_r)|$ restricts to the identity over the interior of $|\Delta[r]|$.
\end{proof}

\begin{notation} \label{not_restricting_functions_paths}
Let $L_i = \im(z_i) \times \im(w_i)$, for each $0 \le i \le r$, and let
$\phi_i \: F_i \to F_{i-1}$ be the order-preserving function $\gamma
\mapsto \gamma \cap L_{i-1}$, for each $0 < i \le r$.
\end{notation}

As before, these notations depend on the implicit choice of a non-degenerate
$r$-simplex $(z, w)$.  To conclude this section, we clarify how the
partially ordered set $F_0 \cup \dots \cup F_r$ is determined by the
diagram
$$
F_r \xrightarrow{\phi_r} \dots \xrightarrow{\phi_1} F_0
$$
of partially ordered sets and order-preserving functions.

\begin{definition}[\cite{WJR}*{2.4.3}]
	\label{def_mappingcylinder_partiallyorderedsets}
Given an order-preserving function $\varphi \: V \to W$ of partially
ordered sets, let $P(\varphi) = V \sqcup_\varphi W$ be the disjoint
union of the sets $V$ and~$W$, with the partial ordering generated by the
relations in $V$, the relations in~$W$, and the relation $\varphi(v) < v$
for all $v \in V$.

In other words, for $v \in V$ and $w \in W$ we have $w \le v$
in $P(\varphi)$ if and only if $w \le \varphi(v)$ in $W$.  We write
$\varphi \vee 1 \: P(\varphi) \to W$ for the order-preserving function
that restricts to $\varphi$ on $V$ and to the identity on $W$, and $\pi \:
P(\varphi) \to [1] = \{0 < 1\}$ for the order-preserving function that
maps $V$ to $1$ and $W$ to $0$.
\end{definition}

\begin{definition}[\cite{WJR}*{2.4.13}]
	\label{def_iter_mappingcylinder_partiallyorderedsets}
For a sequence of $r$ order-preserving functions
$$
V_r \xrightarrow{\varphi_r} \dots \xrightarrow{\varphi_1} V_0
$$
let $P(\varphi_r, \dots, \varphi_1) = V_r \sqcup_{\varphi_r} \dots
\sqcup_{\varphi_1} V_0$ be the disjoint union of the sets $V_0, \dots,
V_r$, with the partial ordering generated by the relations in $V_i$ for
$0\le i\le r$, and the relations $\varphi_i(v) < v$ for all $v \in V_i$
and $0<i\le r$.
We write $\pi \: P(\varphi_r, \dots, \varphi_1) \to [r]$ for the
order-preserving function that takes $V_i$ to $i$, for each $0 \le i \le r$.

Alternatively, for $r\ge2$ we can construct $P(\varphi_r, \dots,
\varphi_1)$ as an instance $P(\psi_1)$ of the previous definition, by
an induction on the length of the sequence of functions.  To start the
induction, let $\psi_r = \varphi_r$.  Thereafter, for $1 \le j < r$,
let $\psi_j$ be the composite
$$
P(\varphi_r, \dots, \varphi_{j+1}) \xrightarrow{\psi_{j+1} \vee 1}
V_j \xrightarrow{\varphi_j} V_{j-1} \,.
$$
Then $P(\varphi_r, \dots, \varphi_j) = P(\psi_j)$ as partially ordered
sets.
\end{definition}

\begin{lemma} \label{lem_xi_pi_coincide}
Let $F_r \xrightarrow{\phi_r} \dots \xrightarrow{\phi_1} F_0$ be as in
Notation~\ref{not_restricting_functions_paths}.  Then $F_0 \cup \dots
\cup F_r = P(\phi_r, \dots, \phi_1)$ as partially ordered sets, and
the functions $\xi \: F_0 \cup \dots \cup F_r \to [r]$ and
$\pi \: P(\phi_r, \dots, \phi_1) \to [r]$ coincide.
\end{lemma}

\begin{proof}
The disjoint unions $F_0 \cup \dots \cup F_r$ and $F_r \sqcup_{\phi_r}
\dots \sqcup_{\phi_1} F_0$ agree as subsets of $C = (\Delta[m] \times
\Delta[n])^\#$.  The lemma asserts that the subset partial ordering on
$F_0 \cup \dots \cup F_r$ inherited from $C$ equals the partial ordering
generated by the subset partial orderings on the individual subsets $F_i$
for $0 \le i \le r$, together with the relations $\phi_i(\gamma) < \gamma$
for $\gamma \in F_i$ and $0 < i \le r$.

We prove that $F_{j-1} \cup \dots \cup F_r$ and $P(\phi_r, \dots,
\phi_j) = F_r \sqcup_{\phi_r} \dots \sqcup_{\phi_j} F_{j-1}$ agree as
partially ordered subsets of $C$, by a descending induction on $j$.
This is clear for $j=r+1$, when we interpret both expressions as $F_r$.
Thereafter, for $1 \le j \le r$, we can inductively assume that
$F_j \cup \dots \cup F_r = P(\phi_r, \dots, \phi_{j+1})$ as partially
ordered subsets of $C$.

The order-preserving function $\psi_j \: P(\phi_r, \dots, \phi_{j+1})
\to F_{j-1}$ is then equal to the function $F_j \cup \dots \cup F_r \to
F_{j-1}$ given by $\gamma \mapsto \gamma \cap L_{j-1}$.  The partial
ordering on $P(\phi_r, \dots, \phi_j) = P(\psi_j)$ agrees with the subset
partial ordering when restricted to $F_j \cup \dots \cup F_r = P(\phi_r,
\dots, \phi_{j+1})$, and when restricted to $F_{j-1}$.
It remains to check that, for $\gamma \in F_j \cup \dots \cup F_r$
and $\beta \in F_{j-1}$, we have $\beta \le \gamma$ in the subset
partial ordering on $F_{j-1} \cup (F_j \cup \dots \cup F_r)$ if and only if
$\beta \le \gamma$ in $P(\psi_j)$.  But this is clear, since
$\beta \le \gamma$ in $P(\psi_j)$ if and only if $\beta \le \psi_j(\gamma)$
in $F_{j-1}$, and $\psi_j(\gamma) = \gamma \cap L_{j-1}$ is the greatest 
face of $\gamma$ that lies in $F_{j-1}$.
\end{proof}

\begin{definition}[\cite{WJR}*{2.1.9}]
A map $f \: A \to B$ of finite simplicial sets is \emph{simple
over $U$}, for a given subset $U \subseteq |B|$, if for each $b \in U$
the preimage $|f|^{-1}(b)$ is contractible.
\end{definition}

To summarize what we have learned so far, let $\kappa$ be as
in Proposition~\ref{prop_main_result}, let $(z, w)$ be as in
Notation~\ref{not_zw_posets_Pi}, and let $\pi \: P(\phi_r, \dots, \phi_1)
\to [r]$ be as in Notation~\ref{not_restricting_functions_paths} and
Definition~\ref{def_iter_mappingcylinder_partiallyorderedsets}.

\begin{lemma} \label{lemma_kappa_simple_if_Npi_is}
The map $\kappa$ is simple over the open $r$-cell corresponding to $(z,
w)$ in the CW structure of $|\Sd \Delta [m]\times \Sd \Delta [n]|$,
if and only if $N\pi$ is simple over the interior of $|\Delta [r]|$.
\end{lemma}

\section{Mapping cylinders} \label{sec_map_cyl}

\noindent
The (backward) reduced mapping cylinder, as introduced in
\cite{WJR}*{\S2.4}, is a functor that to each simplicial map $f \: X \to
Y$ associates a simplicial set $M(f)$, factoring $f$ as an inclusion
$\In \: X \to M(f)$ followed by a simple projection map $\pr \: M(f)
\to Y$.  There is a natural reduction map $\red \: T(f) \to M(f)$ from
the ordinary mapping cylinder $T(f) = X \times \Delta[1] \cup_X Y$ to the
reduced mapping cylinder, and the ordinary coordinate projection $T(f)
\to \Delta[1]$ factors naturally through a reduced coordinate projection
$\pi \: M(f) \to \Delta[1]$:
$$
\xymatrix{
X \ar@{ >->}[d] \ar@{ >->}[dr]^-{\In} \ar@(r,u)[drr]^-{f} \\
T(f) \ar[r]^-{\red} \ar[dr]
	& M(f) \ar[r]^-{\pr}_-{\simeq_s} \ar[d]^-{\pi} & Y \\
& \Delta[1] \,.
}
$$
More generally, given a sequence $X_r \xrightarrow{f_r} \dots
\xrightarrow{f_1} X_0$ of maps of simplicial sets, there is an 
$r$-fold iterated reduced mapping cylinder $M(f_r, \dots, f_1)$,
a reduction map $\red \: T(f_r, \dots, f_1) \to M(f_r,
\dots, f_1)$ from the $r$-fold iterated ordinary mapping cylinder,
and compatible coordinate projection maps:
$$
\xymatrix{
T(f_r, \dots, f_1) \ar[r]^-{\red} \ar[dr] &
	M(f_r, \dots, f_1) \ar[d]^-{\pi} \\
& \Delta[r] \,.
}
$$
In this section, we shall recognize the map $N\pi \: N(P(\phi_r, \dots,
\phi_1)) \to \Delta[r]$ as the reduced coordinate projection map $\pi \:
M(f_r, \dots, f_1) \to \Delta[r]$ for the sequence of simplicial sets
$X_i = NF_i$ and maps $f_i = N\phi_i$.  Thereafter, we use the second
author's criterion \cite{WJR}*{2.4.16} to show that the iterated reduction
map $T(f_r, \dots, f_1) \to M(f_r, \dots, f_1)$ is simple in this case.
This reduces the problem of showing that $\pi$ is simple to the problem
of showing that the classifying spaces $|NF_i|$ are contractible, which
we will address in the following, final, section.

\begin{definition}[\cite{WJR}*{2.4.5}]
Given an order-preserving function $\varphi \: V \to W$ of partially
ordered sets, let $\In \: V \to P(\varphi) = V \sqcup_{\varphi} W$ and
$\In' \: W \to P(\varphi)$ be the front and back inclusions, respectively,
and let $\pr = \varphi \vee 1 \: P(\varphi) \to W$ and $\pi \: P(\varphi)
\to [1]$ be as before.  Consider the commutative diagram
$$
\xymatrix{
& V \ar[dl]_-{i_1} \ar[dr]^-{\In} \ar@(r,u)[drr]^-{\varphi} \\
V \times [1] \ar[rr]^-{\rho} && P(\varphi) \ar[r]^-{\pr} \ar[d]^-{\pi} & W \\
V \ar[u]^-{i_0} \ar[r]^-{\varphi} & W \ar[ur]^-{\In'} & [1]
}
$$
of order-preserving functions, where $i_t(v) = (v, t)$, $\rho(v, 1) =
\In(v)$ and $\rho(v, 0) = \In'(\varphi(v))$.  Applying nerves, we get
the commutative diagram
$$
\xymatrix{
& NV \ar[dl]_-{i_1} \ar[d] \ar[dr]^-{\In} \ar@(r,u)[drr]^-{N\varphi} \\
NV \times \Delta[1] \ar[r] & T(N\varphi) \ar[r]^-{\red}
	& N(P(\varphi)) \ar[r]^-{\pr} \ar[d]^-{\pi} & NW \\
NV \ar[u]^-{i_0} \ar[r]^-{N\varphi} & NW \ar[u] \ar[ur]^-{\In'} & \Delta[1] \,,
}
$$
where the lower left hand square is the pushout defining $T(N\varphi)
= NV \times \Delta[1] \cup_{NV} NW$, and the reduction map $\red$ is
induced by $N\rho$ and $\In'$.
We define the \emph{reduced mapping cylinder} of the simplicial map $f =
N\varphi \: NV \to NW$ to be the nerve $M(f) = N(P(\varphi))$ of the
partially ordered set $P(\varphi) = V \sqcup_\phi W$, with inclusion,
projection, reduction and coordinate projection maps as in the diagram
above.
\end{definition}

\begin{example}
  \label{ex_reduced_cylinder_for_sigma_01}
Let $[2] = \{0 < 1 < 2\}$ and $[1]' = \{0' < 1'\}$.
The reduced mapping cylinders of the nerves of the elementary degeneracy
operators $\sigma_0 \: [2] \to [1]'$; $0, 1 \mapsto 0'$ (left) and $\sigma_1
\: [2] \to [1]$; $1, 2 \mapsto 1'$ (right) are illustrated in the pictures
below.
$$
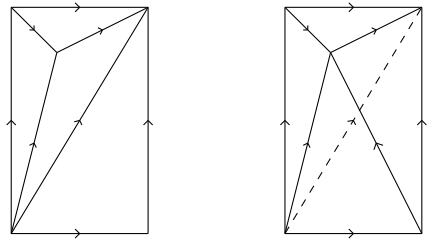
\begin{tikzpicture}[scale=0.6, transform shape]
	\draw (0,5) -- (3,5);
	\draw (1.5,5) -- (1.4,5.1);
	\draw (1.5,5) -- (1.4,4.9);
	\draw (0,5) -- (1,4);
	\draw (0.5,4.5) -- (0.5,4.6);
	\draw (0.5,4.5) -- (0.4,4.5);
	\draw (1,4) -- (3,5);
	\draw (2,4.5) -- (1.9,4.55);
	\draw (2,4.5) -- (1.95,4.4);
	\draw (0,5) -- (0,0);
	\draw (0,2.5) -- (-0.1,2.4);
	\draw (0,2.5) -- (0.1,2.4);
	\draw (0,0) -- (3,0);
	\draw (1.5,0) -- (1.4,0.1);
	\draw (1.5,0) -- (1.4,-0.1);
	\draw (3,0) -- (3,5);
	\draw (3,2.5) -- (2.9,2.4);
	\draw (3,2.5) -- (3.1,2.4);
	\draw (0,0) -- (3,5);
	\draw (1.5,2.5) -- (1.38,2.47);
	\draw (1.5,2.5) -- (1.55,2.4);
	\draw (1,4) -- (0,0);
	\draw (0.5,2) -- (0.4,1.95);
	\draw (0.5,2) -- (0.55,1.9);

	\draw (6,5) -- (9,5);
	\draw (7.5,5) -- (7.4,5.1);
	\draw (7.5,5) -- (7.4,4.9);
	\draw (6,5) -- (7,4);
	\draw (6.5,4.5) -- (6.5,4.6);
	\draw (6.5,4.5) -- (6.4,4.5);
	\draw (7,4) -- (9,5);
	\draw (8,4.5) -- (7.9,4.55);
	\draw (8,4.5) -- (7.95,4.4);
	\draw (6,5) -- (6,0);
	\draw (6,2.5) -- (5.9,2.4);
	\draw (6,2.5) -- (6.1,2.4);
	\draw (6,0) -- (9,0);
	\draw (7.5,0) -- (7.4,0.1);
	\draw (7.5,0) -- (7.4,-0.1);
	\draw (9,0) -- (9,5);
	\draw (9,2.5) -- (8.9,2.4);
	\draw (9,2.5) -- (9.1,2.4);
	\draw [dashed] (6,0) -- (9,5);
	\draw (7.5,2.5) -- (7.37,2.48);
	\draw (7.5,2.5) -- (7.55,2.4);
	\draw (7,4) -- (6,0);
	\draw (6.5,2) -- (6.4,1.95);
	\draw (6.5,2) -- (6.55,1.9);
	\draw (9,0) -- (7,4);
	\draw (8,2) -- (7.95,1.85);
	\draw (8,2) -- (8.15,1.95);
\end{tikzpicture}
$$
The extra relation $1' < 1$ in $P(\sigma_1) = [2] \sqcup_{\sigma_1} [1]'$,
which does not hold in $P(\sigma_0) = [2] \sqcup_{\sigma_0} [1]'$, leads
to the additional $3$-simplex $0' < 1' < 1 < 2$ on the right hand side.
\end{example}

The reduced mapping cylinder $M(f)$ for a general simplicial map $f \: X
\to Y$ can be defined as a left Kan extension from these special cases,
essentially as in \cite{WJR}*{2.4.4}, but the definition just given
suffices for our purposes.  See \cite{WJR}*{2.4.12} for the compatibility
of the two definitions.

We have already alluded to the following result.

\begin{lemma} \label{lem_pr_is_simple}
Let $f \: X \to Y$ be any map of finite simplicial sets.
The projection $\pr \: M(f) \to Y$ is a simple map.
\end{lemma}

\begin{proof}
See \cite{WJR}*{2.4.8}.
\end{proof}

The process of making reduced mapping cylinders can be iterated.

\begin{definition} \label{def_iterated_mapping_cylinder}
Suppose given a sequence
$V_r \xrightarrow{\varphi_r} \dots \xrightarrow{\varphi_1} V_0$
of partially ordered sets and order-preserving functions.  Define the
\emph{iterated reduced mapping cylinder} of the sequence
$$
X_r \xrightarrow{f_r} \dots \xrightarrow{f_1} X_0
$$
of simplicial sets and maps, with $X_i = NV_i$ and $f_i = N\varphi_i$,
to be the nerve
$$
M(f_r, \dots, f_1) = N(P(\varphi_r, \dots, \varphi_1))
$$
of the partially ordered set $P(\varphi_r, \dots, \varphi_1) = V_r
\sqcup_{\varphi_r} \dots \sqcup_{\varphi_1} V_0$.
Let the \emph{reduced coordinate projection}
$\pi \: M(f_r, \dots, f_1) \to \Delta[r]$ be the nerve of the
order-preserving function $\pi \: P(\varphi_r, \dots, \varphi_1)
\to [r]$ that takes $V_i$ to $i$, for each $0 \le i \le r$.
\end{definition}

The iterated reduced mapping cylinder $M(f_r, \dots, f_1)$ of a general
sequence $X_r \to \dots \to X_0$ of simplicial sets and maps can, again,
be defined as a left Kan extension from these special cases, but the
definition given will suffice for this paper.

\begin{lemma}
	\label{lemma_iterated_reduced_as_single_reduced}
The iterated reduced mapping cylinder $M(f_r, \dots, f_1)$ is naturally
isomorphic to the reduced mapping cylinder of the composite map
$$
M(f_r, \dots, f_2) \xrightarrow{\pr} X_1 \xrightarrow{f_1} X_0 \,,
$$
where the iterated projection $\pr \: M(f_r, \dots, f_1) \to X_0$
corresponds to the projection $\pr \: M(f_1 \circ \pr) \to X_0$.
\end{lemma}

\begin{proof}
Recall the notation from
Definition~\ref{def_iter_mappingcylinder_partiallyorderedsets}.
The composite of the two order-preserving functions
$$
P(\varphi_r, \dots, \varphi_2) \xrightarrow{\psi_2 \vee 1}
	V_1 \xrightarrow{\varphi_1} V_0 \,,
$$
equals $\psi_1$.  Hence the composite map $f_1 \circ \pr$ equals
$N\psi_1$, and $M(f_r, \dots, f_1) = N(P(\varphi_r, \dots, \varphi_1))
= N(P(\psi_1))$ equals $M(N\psi_1) = M(f_1 \circ \pr)$, as asserted.
\end{proof}

\begin{definition} \label{def_iterated_ordinary_mapping_cylinder}
Let $X_r \xrightarrow{f_r} \dots \xrightarrow{f_1} X_0$ be any sequence of
simplicial sets and maps.  Define the \emph{iterated ordinary
mapping cylinder} $T(f_r, \dots, f_1) = T(f_1 \circ \pr)$ to be the
ordinary mapping cylinder of the composite map
$$
T(f_r, \dots, f_2) \xrightarrow{\pr} X_1 \xrightarrow{f_1} X_0 \,,
$$
and let the iterated projection $\pr \: T(f_r, \dots, f_1) \to X_0$
be the ordinary projection $\pr \: T(f_1 \circ \pr) \to X_0$.
\end{definition}

\begin{definition}
Suppose given any sequence
$V_r \xrightarrow{\varphi_r} \dots \xrightarrow{\varphi_1} V_0$
of partially ordered sets and order-preserving functions.
Let the \emph{$r$-fold iterated reduction map}
$$
\red \: T(f_r, \dots, f_1) \to M(f_r, \dots, f_1)
$$
be defined as the composite map
$$
T(T(f_r, \dots, f_2) \to X_0)
\to
T(M(f_r, \dots, f_2) \to X_0)
\xrightarrow{\red}
M(M(f_r, \dots, f_2) \to X_0) \,.
$$
Here the left hand map is induced by the $(r-1)$-fold iterated reduction
map $\red \: T(f_r, \dots, f_2) \to M(f_r, \dots, f_2)$ and the identity
map of $X_0$, while the right hand map is the reduction map for $f_1
\circ \pr \: M(f_r, \dots, f_2) \to X_0$.
\end{definition}

\begin{notation}
Consider the terminal sequence $[0] \to \dots \to [0]$
of partially ordered sets and order-preserving functions.
The order-preserving function $\pi \: [0] \sqcup_{\varphi_r} \dots
\sqcup_{\varphi_1} [0] \to [r]$ is a bijection, so the $r$-fold iterated
reduced mapping cylinder
$$
M^r = M(\Delta[0] \to \dots \to \Delta[0])
	= N([0] \sqcup_{\varphi_r} \dots \sqcup_{\varphi_1} [0])
$$
maps isomorphically to $\Delta[r]$ under the reduced coordinate
projection.  Let $T^r = T(\Delta[0] \to \dots \to \Delta[0])$ be
the $r$-fold iterated ordinary mapping cylinder.  Then $T^0 \cong
\Delta[0]$, $T^1 \cong \Delta[1]$, and $T^r \cong T^{r-1} \times \Delta[1]
\cup_{T^{r-1}} \Delta[0]$ for $r\ge2$.  The iterated reduction map $\red
\: T^r \to M^r$ agrees, up to the isomorphism $\pi$, with the ordinary
cylinder projection $T^r \to \Delta[r]$.
\end{notation}

The map $T^r \to M^r$ is an isomorphism for $r = 0$ and~$1$.
To address the case $r\ge2$, the following terminology will be
convenient.

\begin{definition}[\cite{WJR}*{2.4.7}]
	\label{def_homotopy_equivalence_over_target}
A map $f \: A \to B$ of finite simplicial sets is a \emph{homotopy
equivalence over the target} if it has a section $s \: B \to A$ such
that $|sf|$ is homotopic to the identity map on $|A|$, by a homotopy
over $|B|$.  The homotopy provides a contraction of each point inverse
of $|f|$, so such a map $f$ is simple.
\end{definition}

\begin{lemma}
  \label{lemma_cylinder_reduction_between_terminal_ordcyl_and_redcyl}
The iterated reduction map $\red \: T^r \to M^r$ is simple.
\end{lemma}

\begin{proof}
By induction, we may assume that $r\ge2$ and that $T^{r-1} \to M^{r-1}$
is simple.  We must prove that the composite
$$
T(T^{r-1} \to \Delta[0]) \to T(M^{r-1} \to \Delta[0])
\xrightarrow{\red} M(M^{r-1} \to \Delta[0])
$$
is simple.  The left hand map is simple by the inductive hypothesis and
the gluing lemma for simple maps \cite{WJR}*{2.1.3(d)}.  The right hand
map is isomorphic to the reduction map for the unique map $N\epsilon \:
\Delta[r-1] \to \Delta[0]$, which is the canonical map from the pushout
to the lower right hand entry in the commutative square
$$
\xymatrix{
\Delta[r-1] \ar[r]^-{N\epsilon} \ar[d]_-{i_0}
	& \Delta[0] \ar[d]^-{\In'} \\
\Delta[r-1] \times \Delta[1] \ar[r]^-{N\rho}
	& N([r-1] \sqcup_{\epsilon} [0]) \,.
}
$$
By the gluing lemma, and the fact that $N\epsilon$ is simple, it suffices
to prove that $N\rho$ is simple.  There is an order-preserving bijection
$[r-1] \sqcup_{\epsilon} [0] \cong [r]$, taking $0 \in [0]$ to $0 \in
[r]$ and $i \in [r-1]$ to $i+1 \in [r]$.
Under this identification, $N\rho$ corresponds to $N\theta$, where
$\theta \: [r-1] \times [1] \to [r]$ is given by $\theta(i,0) = 0$
and $\theta(i,1) = i+1$.

To prove that $N\theta \: \Delta[r-1] \times \Delta[1] \to \Delta[r]$
is simple, it suffices to show that it is a homotopy equivalence over
the target.  The order-preserving function $\sigma \: [r] \to [r-1]
\times [1]$, given by $\sigma(0) = (0,0)$ and $\sigma(j) = (j-1,1)$ for $1
\le j \le r$, is a section of $\theta$, and $(\sigma \circ \theta)(i,t)
\le (i,t)$ for all $(i,t) \in [r-1] \times [1]$.  Hence $N\sigma$ is a
section of $N\theta$, and there is a simplicial homotopy from $N\sigma
\circ N\theta$ to the identity, covering the identity of $\Delta[r]$.
\end{proof}

\begin{remark} \label{remark_outline_of_proof}
Naturality of the iterated reduction map, with respect to the
terminal morphism
$$
\xymatrix{
X_r \ar[r]^-{f_r} \ar[d] & \dots \ar[r]^-{f_1} & X_0 \ar[d] \\
\Delta[0] \ar[r] & \dots \ar[r] & \Delta[0]
}
$$
in the category of $r$-tuples of composable maps, leads to the commutative
diagram
$$
\xymatrix{
T(f_r, \dots, f_1) \ar[r]^-{\red} \ar[d] & M(f_r, \dots, f_1) \ar[d]
	\ar[dr]^-{\pi} \\
T^r \ar[r]^-{\red}_-{\simeq_s} & M^r \ar[r]_-{\cong} & \Delta[r] \,.
}
$$
The lower reduction map $T^r \to M^r$ is simple, by the previous lemma.
Each point inverse of the geometric realization $|T(f_r, \dots, f_1)|
\to |T^r|$ of the left hand vertical map is homeomorphic to one of the
spaces $|X_0|, \dots, |X_r|$.  This is obvious for $r=0$ and~$1$, and
follows by an easy induction for $r\ge2$.  Hence the map $T(f_r, \dots,
f_1) \to T^r$ is simple if (and only if) these spaces are contractible.

The left hand picture in Example~\ref{ex_reduced_cylinder_for_sigma_01}
illustrates that the point inverses of the map $|\pi| \: |M(f_r, \dots,
f_1)| \to |\Delta[r]|$ need not be homeomorphic to one of the spaces
$|X_0|, \dots, |X_r|$.  This is why we introduce ordinary mapping
cylinders and reduction maps, as technical tools.

In the remainder of this section we shall prove that the upper
reduction map $T(f_r, \dots, f_1) \to M(f_r, \dots, f_1)$ is
simple in the case when $X_i = NF_i$ and $f_i = N\phi_i$ are as in
Notation~\ref{not_restricting_functions_paths}.  In the following
section we shall prove that the classifying spaces $|NF_i| = |X_i|$
are contractible in that case.  By the diagram above, and the right
cancellation property for simple maps, it will then follow that
$\pi \: M(f_r, \dots, f_1) \to \Delta[r]$ is simple.  In view of
Lemma~\ref{lemma_kappa_simple_if_Npi_is}, this will complete the proof
of our main theorem.
\end{remark}

The reduction map $\red \: T(f) \to M(f)$ is not always simple, as
the restriction $f \: \partial\Delta[2] \to \Delta[1]$ of $N\sigma_0
\: \Delta[2] \to \Delta[1]$ illustrates; see the left hand figure
in Example~\ref{ex_reduced_cylinder_for_sigma_01}.  When $f =
N\varphi \: NV \to NW$ is the nerve of an order-preserving function
of partially ordered sets, we have a useful criterion explained in
Proposition~\ref{prop_criterion_simple_cylinder_reduction} below.

\begin{definition}
Let $V$ be a partially ordered set.  A subset $I \subseteq V$ is called
a \emph{left ideal} if $v \in I$ and $u \le v$ in $V$ implies $u \in I$.
It is called a \emph{right ideal} if $v \in I$ and $v \le w$ in $V$
implies $w \in I$.  We always give $I$ the subset partial ordering.
\end{definition}

\begin{definition}
Suppose given a partially ordered set $V$ and an element $v \in V$.
Let $V/v$ denote the left ideal in $V$ consisting of all elements $u \in
V$ with $u \le v$.  Any order-preserving function $\varphi \: V \to W$
restricts to an order-preserving function $\varphi/v \: V/v \to W/\varphi(v)$.
\end{definition}

\begin{proposition}
	\label{prop_criterion_simple_cylinder_reduction}
Let $\varphi \: V \to W$ be an order-preserving function of finite
partially ordered sets. Then the following conditions are equivalent:
\begin{enumerate}
\item
The nerve of $\varphi/v \: V/v \to W/\varphi(v)$ is simple, for each
element $v\in V$.
\item
The reduction map $\red \: T(N\varphi) \to M(N\varphi)$ is simple.
\end{enumerate}
\end{proposition}

\begin{proof}
See \cite{WJR}*{2.4.16}.
\end{proof}

\begin{lemma}
	\label{lemma_phi_vee_1_simple_cylinder_reduction}
Let $\varphi \: V \to W$ and $\pr = \varphi \vee 1 \: P(\varphi) = V
\sqcup_{\varphi} W \to W$ be as above.  The nerve of the order-preserving
function $\pr/v \: P(\varphi)/v \to W/\pr(v)$ is simple, for each $v
\in P(\varphi)$.
\end{lemma}

\begin{proof}
If $v \in W \subseteq P(\varphi)$ then $P(\varphi)/v =
W/v$, $\pr(v) = v$ and $\pr/v \: W/v \to W/v$ is the identity,
whose nerve is obviously simple.

Otherwise, if $v \in V \subseteq P(\varphi)$ then $P(\varphi)/v = (V
\sqcup_{\varphi} W)/v$ is equal to $P(\varphi/v) = V/v \sqcup_{\varphi/v}
W/\varphi(v)$, and $\pr(v) = \varphi(v)$, so we can identify $\pr/v$
with the order-preserving function $\varphi/v \vee 1 \: P(\varphi/v)
\to W/\varphi(v)$.  Its nerve is the projection map from the reduced
mapping cylinder of $N(\varphi/v)$ to its target, which is simple by
\cite{WJR}*{2.4.8}, recalled above as Lemma~\ref{lem_pr_is_simple}.
\end{proof}

We now return to the context of the previous section, where $(z_0, w_0)
< \dots < (z_r, w_r)$ is a non-degenerate $r$-simplex in $\Sd \Delta[m]
\times \Sd \Delta[n]$, $L_i = \im(z_i) \times \im(w_i) \subset [m] \times
[n]$, $F_i$ is the set of $(z_i, w_i)$-paths that are $(z_j, w_j)$-full
for every $0 \le j < i$, and $\phi_i \: F_i \to F_{i-1}$ maps $\gamma$
to $\phi_i(\gamma) = \gamma \cap L_{i-1}$.

\begin{definition}
The bijective correspondence between the elements $\gamma \:
[k] \to [m] \times [n]$ of $C = (\Delta[m] \times \Delta[n])^\#$ and
the non-empty, totally ordered subsets $\im(\gamma) \subseteq [m] \times
[n]$ lets us define two further operations:

First, there is an order-preserving function $C/\gamma \times C/\gamma
\to C/\gamma$ that takes two paths $\alpha, \beta \in C/\gamma$ to
the path $\alpha \cup \beta \in C/\gamma$ with $\im(\alpha \cup \beta)
= \im(\alpha) \cup \im(\beta)$.

Second, there is an order-preserving function $F_i \to C$ taking $\beta
\in F_i$ to the path $\beta \setminus L_{i-1}$ with image $\im(\beta)
\setminus L_{i-1}$.
\end{definition}

\begin{lemma} \label{lemma_mapsinsequence_have_simplecylred}
The nerve of the order-preserving function
$$
\phi_i/\gamma \: F_i/\gamma \to F_{i-1}/\phi_i(\gamma)
$$
is simple, for each $\gamma \in F_i$.  Therefore the reduction map
$T(N\phi_i) \to M(N\phi_i)$ is simple, for each $1 \le i \le r$.
\end{lemma}

\begin{proof}
We prove that $N(\phi_i/\gamma)$ is a homotopy
equivalence over the target, in the sense of
Definition~\ref{def_homotopy_equivalence_over_target}.  Using the
operations mentioned above we can define an order-preserving function
$$
\sigma \: F_{i-1}/\phi_i(\gamma) \to F_i/\gamma
$$
by $\alpha \mapsto \alpha \cup (\gamma \setminus L_{i-1})$.  It is
straightforward to check that $\sigma$ is a section of $\phi_i/\gamma$.
If $\beta \in F_i/\gamma$, we see that
$$
\beta = \phi_i(\beta) \cup (\beta \setminus L_{i-1})
  \le \phi_i(\beta) \cup (\gamma \setminus L_{i-1})
  = (\sigma \circ \phi_i/\gamma)(\beta) \,.
$$
Hence there is a simplicial homotopy from the identity of $N(F_i/\gamma)$
to $N\sigma \circ N(\phi_i/\gamma)$, covering the identity of
$N(F_{i-1}/\phi_i(\gamma))$.

The second assertion now follows from
Proposition~\ref{prop_criterion_simple_cylinder_reduction}.
\end{proof}

\begin{lemma}
	\label{lemma_cylinder_reduction_map_is_simple}
The $r$-fold iterated reduction map
$$
\red \: T(N\phi_r, \dots, N\phi_1) \to M(N\phi_r, \dots, N\phi_1)
$$
is simple, for each $r\ge0$.
\end{lemma}

\begin{proof}
This is trivial for $r=0$, and was proved in
Lemma~\ref{lemma_mapsinsequence_have_simplecylred} for $r=1$.  We may
therefore suppose $r\ge2$, and assume, for some $1 \le i \le
r$, that the $(r-i)$-fold iterated reduction map $\red \: T(N\phi_r,
\dots, N\phi_{i+1}) \to M(N\phi_r, \dots, N\phi_{i+1})$ is simple.
We must prove that the $(r-i+1)$-fold iterated reduction map $\red \:
T(N\phi_r, \dots, N\phi_i) \to M(N\phi_r, \dots, N\phi_i)$ is simple.
By definition, this is the composite of two further maps,
each of which will be shown to be simple.

The first of the two is the map of ordinary mapping cylinders
$$
T(T(N\phi_r, \dots, N\phi_{i+1}) \to NF_{i-1}) \to
T(M(N\phi_r, \dots, N\phi_{i+1}) \to NF_{i-1})
$$
induced by the $(r-i)$-fold iterated reduction map and the identity
on $NF_{i-1}$.  It is simple by the inductive hypothesis and the gluing
lemma for simple maps.

The second of the two is the reduction map 
$$
T(M(N\phi_r, \dots, N\phi_{i+1}) \to NF_{i-1}) \to
M(M(N\phi_r, \dots, N\phi_{i+1}) \to NF_{i-1})
$$
associated to $N\phi_i \circ \pr \: M(N\phi_r, \dots, N\phi_{i+1})
\to NF_{i-1}$.
The latter map equals the nerve of the composite order-preserving function
$$
P(\phi_r, \dots, \phi_{i+1})
  \xrightarrow{\psi_{i+1} \vee 1} F_i \xrightarrow{\phi_i} F_{i-1} \,,
$$
which we denote as $\psi_i = \phi_i \circ (\psi_{i+1} \vee 1)$.  Recall from
the proof of Lemma~\ref{lemma_iterated_reduced_as_single_reduced}
that $P(\phi_r, \dots, \phi_{i+1}) = P(\psi_{i+1})$, where $\psi_{i+1}
\: F_r \sqcup_{\phi_r} \dots \sqcup_{\phi_{i+2}} F_{i+1} \to F_i$
is given by the composite $\phi_{i+1} \circ \dots \circ \phi_j$
on $F_j$, so that $\psi_{i+1} \vee 1 = \pr \: P(\psi_{i+1}) \to F_i$
and $\psi_i = \phi_i \circ \pr$.  We will use the criterion in
Proposition~\ref{prop_criterion_simple_cylinder_reduction} to show that
$\red \: T(N\psi_i) \to M(N\psi_i)$ is simple.

Consider any $\gamma \in P(\phi_r, \dots, \phi_{i+1}) =
P(\psi_{i+1})$.  We must show that the nerve of $\psi_i/\gamma$
is simple.  This order-preserving function is the composite of
$\pr/\gamma$ and $\phi_i/\pr(\gamma)$.  The nerve of $\pr/\gamma
\: P(\psi_{i+1})/\gamma \to F_i/\pr(\gamma)$ is simple by
Lemma~\ref{lemma_phi_vee_1_simple_cylinder_reduction}, for the
order-preserving function $\psi_{i+1}$ and the element $\gamma
\in P(\psi_{i+1})$.  The nerve of $\phi_i/\pr(\gamma)$ is simple by
Lemma~\ref{lemma_mapsinsequence_have_simplecylred}, for the element
$\pr(\gamma) \in F_i$.  Hence the composite $N(\phi_i/\pr(\gamma))
\circ N(\pr/\gamma)$ is also simple, as we needed to prove.
\end{proof}

\section{Contracting sets of paths} \label{sec_contracting}

\noindent
As in the previous sections, we consider the map of nerves $\kappa \:
\Sd(\Delta[m] \times \Delta[n]) \to \Sd \Delta[m] \times \Sd \Delta[n]$
induced by the order-preserving function $(\pr_1^\#, \pr_2^\#) \:
(\Delta[m] \times \Delta[n])^\# \to \Delta[m]^\# \times \Delta[n]^\#$.
We consider an $r$-simplex $(z,w)$ in the target of $\kappa$, represented
by a chain $(z_0 \le \dots \le z_r, w_0 \le \dots \le w_r)$, where the
$z_i$ and $w_i$ are faces of $\Delta[m]$ and $\Delta[n]$, respectively.
However, unlike in the previous sections, we no longer assume that
$(z,w)$ is non-degenerate.  This added generality will be convenient
for our inductive proofs.

\begin{notation} \label{not_pzw}
For any $r$-simplex $(z, w) = (z_0 \le \dots \le z_r, w_0 \le \dots \le
w_r)$ in the nerve of the partially ordered set $\Delta[m]^\# \times
\Delta[n]^\#$, let $P^{z,w}$ be the partially ordered set of
$(z_r, w_r)$-paths $\gamma \: [k] \to [m] \times [n]$ that
are $(z_j, w_j)$-full for each $0 \le j < r$, partially ordered as
a subset of $(\Delta[m] \times \Delta[n])^\#$.
\end{notation}

When $(z, w)$ is non-degenerate, $P^{z,w}$ is equal to the partially
ordered set $F_r$ from Notation~\ref{not_partiallyorderedsetsofpaths}.
More generally, for each $0 \le i \le r$ the partially ordered set
$F_i$ is equal to $P^{z',w'}$, where $(z',w') = (z_0 \le \dots \le z_i,
w_0 \le \dots \le w_i)$ is the front $i$-face of $(z,w)$.  As explained in
Remark~\ref{remark_outline_of_proof}, the task that remains is to prove
that each classifying space $|NF_i|$ is contractible.  Hence it will
suffice to prove that $|NP^{z,w}|$ is contractible, for each simplex
$(z,w)$.  We shall prove this by induction on the dimension $r\ge0$
of that simplex.

We begin with the case $r=0$, when $(z, w) = (z_0, w_0) = (\mu, \nu)$ for
some faces $\mu \in \Delta[m]^\#$ and $\nu \in \Delta[n]^\#$.  The proof
we present for the following proposition introduces some ideas that will
reappear in the cases $r\ge1$.

\begin{proposition} \label{prop_npmunu_contractible}
The classifying space $|NP^{\mu,\nu}|$ is contractible, for each
pair of faces $\mu \in \Delta[m]^\#$ and $\nu \in \Delta[n]^\#$.
\end{proposition}

\begin{proof}
Let $\iota_m \: [m] \to [m]$ and $\iota_n \: [n] \to [n]$ be the identity
morphisms.  In this proof it will be convenient to refer to $(\iota_m,
\iota_n)$-paths as \emph{$(m,n)$-paths}, and to write $P^{m,n} =
P^{\iota_m,\iota_n}$ for the partially ordered set formed by these paths.
It consists of the injective order-preserving functions $\gamma \: [k] \to
[m] \times [n]$ such that $\pr_1 \circ \gamma \: [k] \to [m]$ and $\pr_2
\circ \gamma \: [k] \to [n]$ are surjective, and is partially ordered so
that $\beta \le \gamma$ if and only if $\im(\beta) \subseteq \im(\gamma)$.

Represent the faces $\mu$ and $\nu$ as injective order-preserving
functions $\mu \: [m_1] \to [m]$ and $\nu \: [n_1] \to [n]$, respectively.
Then there is an order-preserving bijection $P^{m_1,n_1} \cong P^{\mu,
\nu}$, taking $\gamma \: [k] \to [m_1] \times [n_1]$ to $(\mu \times
\nu)\gamma \: [k] \to [m] \times [n]$.  It will therefore suffice to prove
that the classifying space of $P^{m_1, n_1}$ is contractible, for each
$m_1, n_1\ge0$.  Changing the notation, we will prove that each partially
ordered set $P^{m,n}$ of $(m,n)$-paths has contractible classifying space.
This will be done by induction on $m\ge0$.

Note first that any $(m,n)$-path $\gamma \: [k] \to [m] \times [n]$
must pass through $\gamma(0) = (0,0)$ and $\gamma(k) = (m,n)$, due
to the assumption that $\pr_1 \circ \gamma$ and $\pr_2 \circ \gamma$
are surjective.

For $m=0$, $P^{0,n}$ consists of a single path $[n] \to [0] \times [n]$,
taking $i$ to $(0,i)$, so $|NP^{0,n}|$ is a single point.

For $m=1$, $P^{1,n}$ consists of $n$ ``short'' paths $[n] \to [1]
\times [n]$ (for each $0 \le j < n$ there is one such path mapping $j$
and $j+1$ to $(0,j)$ and $(1,j+1)$, respectively), and $n+1$ ``long''
paths $[n+1] \to [1] \times [n]$ (for each $0 \le j \le n$ there is one
such path mapping $j$ and $j+1$ to $(0,j)$ and $(1,j)$, respectively).
These are partially ordered as a zig-zag of length $2n$, as illustrated
below for $n=2$.
$$
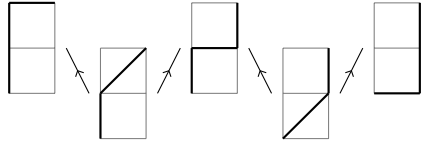
\begin{tikzpicture}[scale=0.6, transform shape]
	\draw [step=1.0cm,gray,thin] (0,1) grid (1,3);
	\draw [step=1.0cm,gray,thin] (2,0) grid (3,2);
	\draw [step=1.0cm,gray,thin] (4,1) grid (5,3);
	\draw [step=1.0cm,gray,thin] (6,0) grid (7,2);
	\draw [step=1.0cm,gray,thin] (8,1) grid (9,3);

	\draw (1.75,1) -- (1.25,2);
	\draw (1.5,1.5) -- (1.4553,1.3659);
	\draw (1.5,1.5) -- (1.6341,1.4553);
	\draw (5.75,1) -- (5.25,2);
	\draw (5.5,1.5) -- (5.4553,1.3659);
	\draw (5.5,1.5) -- (5.6341,1.4553);

	\draw (3.25,1) -- (3.75,2);
	\draw (3.5,1.5) -- (3.5447,1.3659);
	\draw (3.5,1.5) -- (3.3659,1.4553);
	\draw (7.25,1) -- (7.75,2);
	\draw (7.5,1.5) -- (7.5447,1.3659);
	\draw (7.5,1.5) -- (7.3659,1.4553);

	\draw [thick] (0,1) -- (0,3);
	\draw [thick] (0,3) -- (1,3);
	\draw [thick] (2,0) -- (2,1);
	\draw [thick] (2,1) -- (3,2);
	\draw [thick] (4,1) -- (4,2);
	\draw [thick] (4,2) -- (5,2);
	\draw [thick] (5,2) -- (5,3);
	\draw [thick] (6,0) -- (7,1);
	\draw [thick] (7,1) -- (7,2);
	\draw [thick] (8,1) -- (9,1);
	\draw [thick] (9,1) -- (9,3);
\end{tikzpicture}
$$
Hence the nerve $NP^{1,n}$ is a union of $2n$ edges, with alternating
orientations, and the classifying space $|NP^{1,n}|$ is homeomorphic to
an interval.

To handle the general case, fix some $m\ge1$ and $n\ge0$, and suppose
inductively that $|NP^{m_1,n_1}|$ is contractible whenever $0 \le m_1
\le m$ and $n_1 \ge 0$.  We shall prove that the classifying space of
$P^{m+1,n}$ is contractible.

For each $0 \le j \le n$, let $Q_j \subset P^{m+1,n}$ be the
partially ordered subset consisting of the paths $\gamma \: [k] \to
[m+1] \times [n]$ that pass through $(m,j)$, i.e., so that $(m,j) 
\in \im(\gamma)$.  Each $(m+1,n)$-path $\gamma$ has to pass through
some $(m,j)$, since $\pr_2 \circ \gamma \: [k] \to [m+1]$ is
surjective, so the $Q_j$ cover $P^{m+1,n}$.  Furthermore, these
are right ideals in $P^{m+1,n}$, so their nerves $NQ_j$ also cover
$NP^{m+1,n}$, and
$$
|NP^{m+1,n}| = |NQ_0| \cup \dots \cup |NQ_n|
$$
is a finite union of CW complexes.  By a Mayer--Vietoris type argument,
it will therefore suffice to prove that each $\ell$-fold intersection
$$
|NQ_{j_1}| \cap \dots \cap |NQ_{j_\ell}|
	= |N(Q_{j_1} \cap \dots \cap Q_{j_\ell})|
$$
is contractible, for $0 \le j_1 < \dots < j_\ell \le n$ and $\ell\ge1$.
In fact, we only need to consider single and double intersections,
since $Q_{j_1} \cap \dots \cap Q_{j_\ell} = Q_{j_1} \cap Q_{j_\ell}$
in our context.

In the first case, there is an order-preserving bijection
$$
Q_j \cong P^{m,j} \times P^{1,n-j}
$$
for each $0 \le j \le n$, breaking an $(m+1,n)$-path $\gamma$ passing
through $(m,j)$ into two pieces: $\gamma^1$ in $[m] \times [j]$
and $\gamma^2$ in $\{m < m+1\} \times \{j < \dots < n\} \cong
[1] \times [n-j]$.  Hence $|NQ_j| \cong |NP^{m,j}| \times |NP^{1,n-j}|$
is a product of two contractible spaces, by our inductive hypothesis.

In the second case, there is an order-preserving bijection
$$
Q_i \cap Q_j \cong P^{m,i} \times P^{0,j-i} \times P^{1,n-j}
$$
for each $0 \le i < j \le n$, breaking a path $\gamma$ passing through
$(m,i)$ and $(m,j)$ into three pieces: $\gamma^1$ in $[m] \times [i]$,
$\gamma^2$ in $\{m\} \times \{i < \dots < j\} \cong [0] \times [j-i]$
and $\gamma^3$ in $\{m < m+1\} \times \{j < \dots < n\} \cong [1] \times
[n-j]$.  Hence $|N(Q_i \cap Q_j)| \cong |NP^{m,i}| \times |NP^{0,j-i}|
\times |NP^{1,n-j}|$ is a product of three contractible spaces, by the
inductive hypothesis.
\end{proof}

To handle the cases $r\ge1$, we will use the following notation.

\begin{notation}
For $m,n\ge0$, let $C^{m,n} = (\Delta[m] \times \Delta[n])^\#$ be the
partially ordered set of paths $\gamma$, previously denoted $C$, where
$\gamma \: [k] \to [m] \times [n]$ is order-preserving and injective.
Let $D^{m,n} = \Delta[m]^\# \times \Delta[n]^\#$ be the partially
ordered set of pairs of faces $(\mu, \nu)$, where $\mu \: [m_1] \to
[m]$ and $\nu \: [n_1] \to [n]$ are order-preserving and injective.
The order-preserving function $(\pr_1^\#, \pr_2^\#) \: C^{m,n} \to
D^{m,n}$ induces the canonical map $\kappa$ upon passage to nerves.

For any $(p,q) \in [m] \times [n]$, let $C^{m,n}_{(p,q)}$ be the subset of
paths $\gamma \in C^{m,n}$ that pass through $(p,q)$, so that $\gamma(j) =
(p,q)$ for some $j \in [k]$.  Let $D^{m,n}_{(p,q)}$ be the subset of pairs
$(\mu, \nu) \in D^{m,n}$ such that $(p,q) \in \im(\mu) \times \im(\nu)$.
When $(p,q) \in \im(z_0) \times \im(w_0)$, let $P^{z,w}_{(p,q)}$ be the
subset of $(z,w)$-paths $\gamma \in P^{z,w}$ that pass through $(p,q)$,
in the same sense as above.  Each of these subsets is given the subset
partial ordering.
\end{notation}

\begin{lemma} \label{lem_two_parts}
(a)
There is an order-preserving bijection
$$
C^{m,n}_{(p,q)} \cong C^{p,q}_{(p,q)} \times C^{m-p,n-q}_{(0,0)}
$$
taking $\gamma \: [k] \to [m] \times [n]$ with $\gamma(j) = (p,q)$ to
$(\gamma^1, \gamma^2)$, where $\gamma^1 \: [j] \to [p] \times [q]$ and
$\gamma^2 \: [k-j] \to [m-p] \times [n-q]$ are given by $\gamma^1(i) =
\gamma(i)$ for $i \in [j]$ and $\gamma^2(i)+(p,q) = \gamma(i+j)$ for $i
\in [k-j]$.

(b)
There is an order-preserving bijection
$$
D^{m,n}_{(p,q)} \cong D^{p,q}_{(p,q)} \times D^{m-p,n-q}_{(0,0)}
$$
taking $(\mu, \nu)$ to $((\mu^1, \nu^1), (\mu^2, \nu^2))$, where (if $\mu
\: [m_1] \to [m]$ and $\mu(j) = p$) $\mu^1 \: [j] \to [p]$ and $\mu^2 \:
[m_1-j] \to [m-p]$ are given by $\mu^1(i) = \mu(i)$ and $\mu^2(i)+p =
\mu(i+j)$, and similarly for $\nu$, $\nu^1$ and $\nu^2$.

(c)
Suppose that $(p,q) \in \im(z_0) \times \im(w_0)$.
The bijection in~(a) restricts to an order-preserving bijection
$$
P^{z,w}_{(p,q)} \cong P^{z^1,w^1} \times P^{z^2,w^2}
$$
taking a $(z,w)$-path $\gamma$ passing through $(p,q)$ to $(\gamma^1,
\gamma^2)$, where $\gamma^1$ is a $(z^1,w^1)$-path and $\gamma^2$ is
a $(z^2,w^2)$-path.  Here the bijection in~(b) takes $z_i$ to $(z_i^1,
z_i^2)$ and $w_i$ to $(w_i^1, w_i^2)$ for each $i$, and $(z^1, w^1) =
(z_0^1 \le \dots \le z_r^1, w_0^1 \le \dots \le w_r^1)$ and $(z^2, w^2) =
(z_0^2 \le \dots \le z_r^2, w_0^2 \le \dots \le w_r^2)$.
\end{lemma}

\begin{proof}
Part~(a) is clear, since any path $\gamma \: [k] \to [m] \times [n]$
with $\gamma(j) = (p,q)$ must map $[j]$ into $[p] \times [q]$ and
$\{j < \dots < k\} \cong [k-j]$ into $\{p < \dots < m\} \times \{q < \dots
< n \} \cong [m-p] \times [n-q]$.

Part~(b) is also clear, since any face $\mu \: [m_1] \to [m]$ with
$\mu(j) = p$ must map $[j]$ into $[p]$ and $\{j < \dots < k\} \cong
[k-j]$ into $\{p < \dots < m\} \cong [m-p]$, and likewise for~$\nu$.

To prove~(c), note that a path $\gamma$ passing through $(p,q)$ is a
$(z_r, w_r)$-path if and only if $\gamma^1$ (passing through $(p,q)$)
is a $(z_r^1, w_r^1)$-path and $\gamma^2$ (passing through $(0,0)$) is
a $(z_r^2, w_r^2)$-path.  The condition that $\gamma^1$ passes through
$(p,q)$ is automatic, since $(p,q) \in \im(z_0^1) \times \im(w_0^1)
\subseteq \im(z_r^1) \times \im(w_r^1)$, so a $(z_r^1, w_r^1)$-path
in $[p] \times [q]$ must end at $(p,q)$, and similarly for the condition
that $\gamma^2$ passes through~$(0,0)$.

Finally, for each $0 \le j < r$, the $(z_r, w_r)$-path $\gamma$ passing
through $(p,q)$ is $(z_j, w_j)$-full if and only if $\gamma^1$ is $(z_j^1,
w_j^1)$-full and $\gamma^2$ is $(z_j^2, w_j^2)$-full.  This follows from
the result in the previous paragraph, together with the observation that
the path $\gamma \cap \im(z_j) \times \im(w_j)$, with image $\im(\gamma)
\cap \im(z_j) \times \im(w_j)$, passes through $(p,q)$, and corresponds
under the bijection in~(a) with the pair of paths $\gamma^1 \cap
\im(z_j^1) \times \im(w_j^1)$ and $\gamma^2 \cap \im(z_j^2) \times
\im(w_j^2)$.
\end{proof}

\begin{notation}
For $(p,q) \le (p',q')$ in $[m] \times [n]$, let $C^{m,n}_{(p,q),(p',q')}
= C^{m,n}_{(p,q)} \cap C^{m,n}_{(p',q')}$ be the set of paths that pass
through both $(p,q)$ and $(p',q')$.  Let $D^{m,n}_{(p,q),(p',q')} =
D^{m,n}_{(p,q)} \cap D^{m,n}_{(p',q')}$ be the set of pairs $(\mu,
\nu)$ such that $(p,q)$ and $(p',q')$ lie in $\im(\mu) \times
\im(\nu)$.  When $(p,q), (p',q') \in \im(z_0) \times \im(w_0)$, let
$P^{z,w}_{(p,q),(p',q')} = P^{z,w}_{(p,q)} \cap P^{z,w}_{(p',q')}$
be the set of $(z,w)$-paths that pass through both $(p,q)$ and $(p',q')$.
\end{notation}

\begin{lemma} \label{lem_three_parts}
Suppose that $(p,q) \le (p',q')$ in $\im(z_0) \times \im(w_0)$.

(a)
There is an order-preserving bijection
$$
D^{m,n}_{(p,q),(p',q')} \cong D^{p,q}_{(p,q)} \times
	D^{p'-p,q'-q}_{(0,0),(p'-p,q'-q)} \times D^{m-p',n-q'}_{(0,0)}
$$
taking $(\mu,\nu)$ to $((\mu^1,\nu^1), (\mu^2,\nu^2), (\mu^3,\nu^3))$.

(b)
There is an order-preserving bijection
$$
P^{z,w}_{(p,q),(p',q')} \cong P^{z^1,w^1} \times
	P^{z^2,w^2} \times P^{z^3,w^3}
$$
taking the $(z,w)$-paths $\gamma$ that pass through both $(p,q)$
and $(p',q')$, to the triples $(\gamma^1, \gamma^2, \gamma^3)$, where
$\gamma^\ell$ is a $(z^\ell, w^\ell)$-path for each $1\le\ell\le3$.
\end{lemma}

\begin{proof}
This follows by two applications of the previous lemma, or by a direct
proof.  In~(b), $(z^\ell, w^\ell) = (z_0^\ell \le \dots \le z_r^\ell,
w_0^\ell \le \dots \le w_r^\ell)$ for each $1\le\ell\le3$, where $z_i$
corresponds to $(z_i^1, z_i^2, z_i^3)$ as in~(a), and likewise for $w_i$.
\end{proof}

We can now prove our main technical result.

\begin{proposition} \label{prop_npzw_is_contractible}
The classifying space $|NP^{z,w}|$ is contractible,
for each simplex $(z,w)$ in $\Sd \Delta[m] \times \Sd \Delta[n]$.
\end{proposition}

\begin{proof}
The proof involves a double induction.  We begin with an \emph{outer
induction} on the dimension $r\ge0$ of the simplex $(z,w)$.  The initial
case $r=0$ was handled in Proposition~\ref{prop_npmunu_contractible}.
For the outer inductive step, we fix an $r\ge0$ and assume that
$|NP^{z,w}|$ is contractible for each $r$-simplex $(z,w)$.  We shall
prove that $|NP^{x,y}|$ is contractible for each $(r+1)$-simplex
$$
(x,y) = (x_0 \le \dots \le x_{r+1}, y_0 \le \dots \le y_{r+1})
$$
in the nerve of $\Delta[m]^\# \times \Delta[n]^\#$.

Let the pair $(\zeta, \eta)$ be the $0$-th vertex of $(x,y)$, and let
the $r$-simplex $(z,w)$ be the $0$-th face of $(x,y)$.  In other words,
let $\zeta = x_0 \: [s] \to [m]$ and $\eta = y_0 \: [t] \to [n]$, for
some $s,t \ge 0$, and let $z_i = x_{i+1}$ and $w_i = y_{i+1}$ for all
$0 \le i \le r$.  With this notation, an $(x,y)$-path is the same as
a $(z,w)$-path that is $(\zeta,\eta)$-full.
We maintain this notational scheme throughout this proof.

We shall prove that the classifying space of $P^{x,y}$ is contractible,
by an \emph{inner induction} on the dimension $s\ge0$ of the face $\zeta$
of $\Delta[m]$.

\medskip

The inner induction begins with the case $s=0$:
When $\zeta \: [0] \to [m]$ is $0$-dimensional, an $(x,y)$-path is
the same as a $(z,w)$-path that goes through both $(\zeta(0), \eta(0))$
and $(\zeta(0), \eta(t))$.  Hence
$$
P^{x,y} = P^{z,w}_{(\zeta(0), \eta(0)),(\zeta(0), \eta(t))}
\cong P^{z^1,w^1} \times P^{z^2,w^2} \times P^{z^3,w^3}
$$
by Lemma~\ref{lem_three_parts}.  Here $(z_i^1,w_i^1) \in
D^{\zeta(0),\eta(0)}$, $(z_i^2,w_i^2) \in D^{0,\eta(t)-\eta(0)}$ and
$(z_i^3,w_i^3) \in D^{m-\zeta(0),n-\eta(t)}$, for all $0 \le i \le r$.
Each $(z^\ell, w^\ell)$ is an $r$-simplex, for $1 \le \ell \le 3$,
so $|NP^{x,y}|$ is a product of three contractible spaces by the outer
inductive hypothesis, and is therefore contractible.

\medskip

The inner induction continues with the case $s=1$:
When $\zeta \: [1] \to [m]$ is $1$-dimensional, any $(\zeta,\eta)$-full
path must go through $(\zeta(0), \eta(0))$ and $(\zeta(1), \eta(t))$.
Hence
$$
P^{x,y} \cong P^{x^1,y^1} \times P^{x^2,y^2} \times P^{x^3,y^3}
$$
by Lemma~\ref{lem_three_parts}.  Here $(x_i^1,y_i^1)
\in D^{\zeta(0),\eta(0)}$, $(x_i^2,y_i^2) \in
D^{\zeta(1)-\zeta(0),\eta(t)-\eta(0)}$ and $(x_i^3,y_i^3) \in
D^{m-\zeta(1),n-\eta(t)}$, for all $0 \le i \le r+1$.

The part $\zeta^1 = x_0^1$ of $\zeta \: [1] \to [n]$ that lands in
$[\zeta(0)]$ is $0$-dimensional, and similarly for the part $\zeta^3
= x_0^3$.  Hence $P^{x^1,y^1}$ and $P^{x^3,y^3}$ have contractible
classifying spaces, by the already established case $s=0$ of the
inner induction.  It therefore remains to prove that $P^{x^2,y^2}$
has contractible classifying space.  Replacing $(x^2,y^2)$ with $(x,y)$
in the notation, we may and will assume that $\zeta(0) = 0$, $\zeta(1) =
m$, $\eta(0) = 0$ and $\eta(t) = n$, and seek to prove that $P^{x,y}$
has contractible classifying space.

By symmetry, the case $t=0$ can be handled just like the case $s=0$.
We therefore assume $t\ge1$.  For each $0 \le j \le t-1$ let
$$
Q_j = P^{x,y}_{(0,\eta(j)),(m,\eta(j+1))}
$$
be the subset of $P^{x,y}$ of paths that go through both $(0,\eta(j))$
and $(m,\eta(j+1))$, where $\zeta(0) = 0$ and $\zeta(1) = m$.  Each
$(\zeta,\eta)$-full path must go through these two points for some $0
\le j \le t-1$, so $P^{x,y} = Q_0 \cup \dots \cup Q_{t-1}$.  Each $Q_j$
is a right ideal, so $|NP^{x,y}| = |NQ_0| \cup \dots \cup |NQ_{t-1}|$.

We now argue that each $|NQ_j|$ is contractible.  Since every path in
$Q_j$ goes through both $(0,\eta(j))$ and $(m,\eta(j+1))$, we have an
order-preserving bijection
$$
Q_j \cong P^{x^1,y^1} \times P^{x^2,y^2} \times P^{x^3,y^3}
$$
with $(x_i^1,y_i^1) \in D^{0,\eta(j)}$, $(x_i^2,y_i^2) \in
D^{m,\eta(j+1)-\eta(j)}$ and $(x_i^3,y_i^3) \in D^{0,n-\eta(j+1)}$.
Once more, $\zeta^1 = x_0^1$ and $\zeta^3 = x_0^3$ are $0$-dimensional,
so $P^{x^1,y^1}$ and $P^{x^3,y^3}$ have contractible classifying spaces
by the case $s=0$.  On the other hand, $\zeta^2 = x_0^2 = \zeta$ and
$\eta^2 = y_0^2 \: [1] \to [\eta(j+1)-\eta(j)]$ are both $1$-dimensional.

We claim that the classifying space of $P^{x^2,y^2}$ is also
contractible.  Since both $\zeta^2$ and $\eta^2$ are $1$-dimensional,
a $(z^2,w^2)$-path is $(\zeta^2,\eta^2)$-full if and only if it goes
through both $(\zeta^2(0),\eta^2(0)) = (0,0)$ and $(\zeta^2(1),\eta^2(1))
= (m,\eta(j+1)-\eta(j))$.  By assumption, $\zeta^2 = x_0^2 \le x_1^2 =
z_0^2$ and $\eta^2 = y_0^2 \le y_1^2 = w_0^2$, so any $(z_0^2,w_0^2)$-full
path in $[m] \times [\eta(j+1)-\eta(j)]$ will begin at $(0,0)$ and end
at $(m, \eta(j+1)-\eta(j))$.  Hence
$$
P^{x^2,y^2} = P^{z^2,w^2}_{(0,0),(m,\eta(j+1)-\eta(j))} = P^{z^2,w^2} \,,
$$
which has contractible classifying space by the outer inductive
hypothesis, since $(z^2,w^2)$ is an $r$-simplex.  This completes the
proof that each $|NQ_j|$ is contractible.

Next we consider the $Q_i \cap Q_j$ for $0 \le i < j \le t-1$.
If $i+1 < j$ there are no paths that go through
both $(\zeta(1), \eta(i+1))$ and $(\zeta(0), \eta(j))$, so in
these cases $Q_i \cap Q_j = \emptyset$.
This implies that all $\ell$-fold intersections
$Q_{j_1} \cap \dots \cap Q_{j_\ell} = \emptyset$ are empty,
for $0 \le j_1 < \dots < j_\ell \le t-1$ and $\ell\ge3$.

It remains to consider the double intersection $Q_{j-1} \cap Q_j$, for
$1 \le j \le t-1$.  It consists of the $(x,y)$-paths that go through the
four points $(0, \eta(j-1))$, $(0, \eta(j))$, $(m, \eta(j))$ and $(m,
\eta(j+1))$, where $\zeta(0) = 0$ and $\zeta(1) = m$.  There is a unique
such path, with image contained in the totally ordered subset
$$
\{0\} \times [\eta(j)]
	\, \cup \, [m] \times \{\eta(j)\}
	\, \cup \, \{m\} \times \{\eta(j) < \dots < n\}
$$
of $[m] \times [n]$.  Hence $|NQ_{j-1}| \cap |NQ_j|$ is a single
point.

It follows that the union $|NQ_0| \cup \dots \cup |NQ_{t-1}|$ is homotopy
equivalent to the set of points $\{j \mid 0 \le j \le t-1\}$, connected
by the intervals $[j-1,j]$ for $1 \le j \le t-1$.  Their union is the
interval $[0, t-1]$, which is contractible.  This completes the
proof for $s=1$ and $t\ge1$.

\medskip

The inner induction ends with the inner inductive step, for $s\ge2$:
Assume inductively that $|NP^{x',y'}|$ is contractible for each
$(r+1)$-simplex $(x',y')$ such that the dimension of $\zeta' = x'_0$ is
strictly less than $s$.  We must prove that $P^{x,y}$ has contractible
classifying space when the dimension of $\zeta = x_0$, as a face
of $\Delta[m]$, equals $s$.

For each $0 \le j \le t$, where $t$ is the dimension of $\eta = y_0$, let
$$
Q_j = P^{x,y}_{(\zeta(s-1), \eta(j))}
$$
be the partially ordered set of $(x,y)$-paths that go through
$(\zeta(s-1), \eta(j))$.  Each $(x,y)$-path is $(\zeta, \eta)$-full,
hence must pass through one of these points, so $P^{x,y} = Q_0 \cup
\dots \cup Q_t$.  Each $Q_j$ is a right ideal, so $|NP^{x,y}| = |NQ_0|
\cup \dots \cup |NQ_t|$ is a finite union of CW complexes.

By Lemma~\ref{lem_two_parts} there is an order-preserving bijection
$$
Q_j \cong P^{x^1,y^1} \times P^{x^2,y^2}
$$
where $(x^1, y^1)$ and $(x^2, y^2)$ are $(r+1)$-simplices, with $\zeta^1
= x_0^1$ of dimension $s-1$ and $\zeta^2 = x_0^2$ of dimension~$1$.
Hence $|NP^{x^1,y^1}|$ and $|NP^{x^2,y^2}|$ are contractible by the inner
inductive hypothesis, and this implies that $|NQ_j|$ is contractible.

For $0 \le i < j \le t$, the double intersection
$$
Q_i \cap Q_j = P^{x,y}_{(\zeta(s-1),\eta(i)),(\zeta(s-1),\eta(j))}
$$
consists of the $(x,y)$-paths that pass through $(\zeta(s-1),\eta(i))$
and $(\zeta(s-1),\eta(j))$.  By Lemma~\ref{lem_three_parts} there is
an isomorphism
$$
Q_i \cap Q_j \cong P^{x^1,y^1} \times P^{x^2,y^2} \times P^{x^3,y^3}
$$
of partially ordered sets, where $\zeta^1 = x_0^1$ is of
dimension~$s-1$, $\zeta^2 = x_0^2$ is of dimension~$0$, and $\zeta^3 =
x_0^3$ is of dimension~$1$.  Each of the three factors has contractible
classifying space, by the inner inductive hypothesis, so $|NQ_i| \cap
|NQ_j|$ is also contractible.

Finally, for any $0 \le j_1 < \dots < j_\ell \le t$ with $\ell\ge3$,
the $\ell$-fold intersection
$$
Q_{j_1} \cap \dots \cap Q_{j_\ell} = Q_{j_1} \cap Q_{j_\ell}
$$
equals one of the double intersections considered above.  This is
because any $(\zeta,\eta)$-full path that passes through
$(\zeta(s-1), \eta(j_1))$ and $(\zeta(s-1), \eta(j_\ell))$ must
pass through $(\zeta(s-1), \eta(i))$ for each  $j_1 \le i \le j_\ell$.

Hence we have proved that each finite intersection $|NQ_{j_1}| \cap
\dots \cap |NQ_{j_\ell}|$ is contractible, which readily implies that
the union $|NP^{x,y}| = |NQ_0| \cup \dots \cup |NQ_t|$ is contractible.
\end{proof}


\begin{bibdiv}
\begin{biblist}

\bib{Ba}{article}{
   author={Barratt, Michael G.},
   title={Simplicial and semisimplicial complexes},
   note={unpublished manuscript},
   date={1956},
}

\bib{Coh}{book}{
   author={Cohen, Marshall M.},
   title={A course in simple-homotopy theory},
   note={Graduate Texts in Mathematics, Vol. 10},
   publisher={Springer-Verlag},
   place={New York},
   date={1973},
   pages={x+144},
   % review={\MR{0362320 (50 \#14762)}},
}

\bib{FP}{book}{
   author={Fritsch, Rudolf},
   author={Piccinini, Renzo A.},
   title={Cellular structures in topology},
   series={Cambridge Studies in Advanced Mathematics},
   volume={19},
   publisher={Cambridge University Press},
   place={Cambridge},
   date={1990},
   pages={xii+326},
   % isbn={0-521-32784-9},
   % review={\MR{1074175 (92d:55001)}},
   % doi={10.1017/CBO9780511983948},
}

\bib{GZ}{book}{
   author={Gabriel, P.},
   author={Zisman, M.},
   title={Calculus of fractions and homotopy theory},
   series={Ergebnisse der Mathematik und ihrer Grenzgebiete, Band 35},
   publisher={Springer-Verlag New York, Inc., New York},
   date={1967},
   pages={x+168},
   % review={\MR{0210125 (35 \#1019)}},
}

\bib{Kan}{article}{
   author={Kan, Daniel M.},
   title={On c. s. s. complexes},
   journal={Amer. J. Math.},
   volume={79},
   date={1957},
   pages={449--476},
   % issn={0002-9327},
   % review={\MR{0090047 (19,759e)}},
}

\bib{RS}{book}{
   author={Rourke, C. P.},
   author={Sanderson, B. J.},
   title={Introduction to piecewise-linear topology},
   note={Ergebnisse der Mathematik und ihrer Grenzgebiete, Band 69},
   publisher={Springer-Verlag},
   place={New York},
   date={1972},
   pages={viii+123},
   % review={\MR{0350744 (50 \#3236)}},
}

\bib{Spa}{book}{
   author={Spanier, Edwin H.},
   title={Algebraic topology},
   publisher={McGraw-Hill Book Co.},
   place={New York},
   date={1966},
   pages={xiv+528},
   % review={\MR{0210112 (35 \#1007)}},
}

\bib{Th}{article}{
   author={Thomason, R. W.},
   title={Cat as a closed model category},
   journal={Cahiers Topologie G\'eom. Diff\'erentielle},
   volume={21},
   date={1980},
   number={3},
   pages={305--324},
   % issn={0008-0004},
   % review={\MR{591388 (82b:18005)}},
}

\bib{WJR}{book}{
   author={Waldhausen, Friedhelm},
   author={Jahren, Bj{\o}rn},
   author={Rognes, John},
   title={Spaces of PL manifolds and categories of simple maps},
   series={Annals of Mathematics Studies},
   volume={186},
   publisher={Princeton University Press},
   date={2013},
}

\end{biblist}
\end{bibdiv}
\end{document}